\documentclass[12pt]{article}

\usepackage[margin=1.3in, top=1.3in, bottom=1.3in]{geometry}

\usepackage[utf8]{inputenc} 
\usepackage[T1]{fontenc}    
\usepackage{url}            
\usepackage{booktabs}       
\usepackage{amsfonts}       
\usepackage{nicefrac}       
\usepackage{microtype}      

\usepackage{amsmath,amssymb,amsthm,xcolor}
\usepackage{array}
\usepackage{float}

\theoremstyle{definition}

\theoremstyle{plain}
\newtheorem{theorem}{Theorem}
\newtheorem{lemma}{Lemma}
\newtheorem{corollary}{Corollary}

\usepackage{graphicx}
\usepackage{subcaption}
\usepackage{mdframed}
\usepackage{lipsum}
\usepackage{wrapfig}
\usepackage[hidelinks,hypertexnames=false]{hyperref}
\usepackage{enumitem}
\allowdisplaybreaks
\usepackage{epstopdf}
\usepackage{algorithmic}
\usepackage{booktabs}
\usepackage{changepage}
\usepackage{amssymb}

\usepackage{cleveref}
\Crefname{assumption}{assumption}{assumptions}

\crefalias{enumi}{proposition}

\crefname{equation}{}{}
\hypersetup{colorlinks=true,citecolor=blue, linkcolor=blue}

\title{\LARGE Singular perturbation in heavy ball dynamics}

\begin{document}

\author{\large C\'edric Josz\thanks{\url{cj2638@columbia.edu}, IEOR, Columbia University, New York. Research supported by NSF EPCN grant 2023032 and ONR grant N00014-21-1-2282.} \and Xiaopeng Li\thanks{\url{xl3040@columbia.edu}, IEOR, Columbia University, New York.}}
\date{}

\maketitle

\begin{center}
    \textbf{Abstract}
    \end{center}
    \vspace*{-3mm}
 \begin{adjustwidth}{0.2in}{0.2in}
Given a $C^{1,1}_\mathrm{loc}$ lower bounded function $f:\mathbb{R}^n\rightarrow \mathbb{R}$ definable in an o-minimal structure on the real field, we show that the singular perturbation $\epsilon \searrow 0$ in the heavy ball system 
\begin{equation}
\label{eq:P_eps} \tag{$P_\epsilon$}
    \begin{cases}
        \epsilon\ddot{x}_\epsilon(t) + \gamma\dot{x}_\epsilon(t) + \nabla f(x_\epsilon(t)) = 0, ~~~ \forall t \geqslant 0, \\
        x_\epsilon(0) = x_0, ~~~ \dot{x}_\epsilon(0) = \dot{x}_0,
    \end{cases}
\end{equation}
preserves boundedness of solutions, where $\gamma>0$ is the friction and $(x_0,\dot{x}_0) \in \mathbb{R}^n \times \mathbb{R}^n$ is the initial condition. This complements the work of Attouch, Goudou, and Redont which deals with finite time horizons. In other words, this work studies the asymptotic behavior of a ball rolling on a surface subject to gravitation and friction, without assuming convexity nor coercivity.
\end{adjustwidth} 
\vspace*{3mm}

\section{Introduction}
\label{sec:intro}

Let $\|\cdot\|$ be the induced norm of an inner product $\langle \cdot, \cdot\rangle$ on $\mathbb{R}^n$ and let $\nabla f$ denote the gradient of a $C^1$ function $f:\mathbb{R}^n\rightarrow \mathbb{R}$ with respect to $\langle \cdot, \cdot\rangle$. If $\nabla f$ is locally Lipschitz continuous, then we say that $f$ is $C^{1,1}_\mathrm{loc}$. 

Attouch \textit{et al.} \cite[Theorem 5.1]{attouch2000heavy} (see also \cite[Theorem 2.1]{vasil1995boundary}) show that $\epsilon \searrow 0$ in \eqref{eq:P_eps} is a regular perturbation for finite time horizons when $f$ is a $C^{1,1}_\mathrm{loc}$ lower bounded function. In other words, the global solution to \eqref{eq:P_eps} converges uniformly over bounded subsets of $[0,\infty)$ as $\epsilon \searrow 0$ to the global solution of 
\begin{equation}
\label{eq:P_0} \tag{\ensuremath{P_0}}
    \gamma \dot{x}(t) + \nabla f(x(t)) = 0, ~~~ \forall t\geqslant 0, ~~~ x(0) = x_0,
\end{equation} 
where the existence and uniqueness of global solutions follow from \cite[Theorem 3.1 (i)]{attouch2000heavy} and \cite[Proposition 2.3]{santambrogio2017euclidean}. If $x_0$ is near a strict local minimum of $f$ and $\dot{x}_0$ is sufficiently small, then $\epsilon \searrow 0$ becomes a regular perturbation for the infinite time horizon by Hoppensteadt \cite[Theorem]{hoppensteadt1966singular}. This means that uniform convergence holds over the entire set $[0,\infty)$. In particular, the perturbation preserves boundedness of solutions. Otherwise, $\epsilon \searrow 0$ is a singular perturbation for the infinite time horizon \cite[Remark p. 26]{attouch2000heavy} \cite{kokotovic1999singular}, namely, uniform convergence does not hold over $[0,\infty)$. In this manuscript, we nonetheless show the following property.

\begin{theorem}
\label{thm:bounded}
    Let $\gamma>0$ and $f:\mathbb{R}^n\rightarrow\mathbb{R}$ be a $C^{1,1}_\mathrm{loc}$ lower bounded function definable in an o-minimal structure on the real field. For all $(x_0,\dot{x}_0) \in \mathbb{R}^n \times \mathbb{R}^n$, the global solution to \eqref{eq:P_eps} is uniformly bounded for all sufficiently small $\epsilon>0$ if and only if for all $x_0 \in \mathbb{R}^n$, the global solution to \eqref{eq:P_0} is bounded. 
\end{theorem}

In the rest of the manuscript, we fix an arbitrary o-minimal structure on the real field (for e.g., the real field with constants \cite{tarski1951decision,seidenberg1954new}, with restricted analytic functions \cite{gabrielov1996complements}, or with the exponential function \cite{wilkie1996model}) and say that $f$ is definable if it is definable in that structure. This is a common framework for studying gradient systems \cite{bolte2010characterizations,kurdyka43quasi} as it enables ones to harness the Kurdyka-\L{}ojasiewicz inequality \cite{kurdyka1998gradients}\cite[Proposition 1 p. 67]{law1965ensembles}. In particular, for heavy ball with friction, this inequality is used to prove that bounded trajectories have finite length \cite[Theorem 4]{begout2015damped} provided that $f$ is $C^2$. This result also holds with variable friction $\gamma(t) = c_1+c_2/t$, $c_1>0$, $c_2\geqslant 0$ if $\nabla f$ is globally Lipschitz continuous. Theorem \ref{thm:bounded} can thus serve as a criterion to establish convergence in the small mass $\epsilon$ regime when $f$ is not convex. Note that when $f$ is $C^1$ convex and admits a minimum, heavy ball trajectories with friction are automatically bounded \cite{alvarez2000minimizing}.

The manuscript is organized as follows. Section \ref{sec:Example} contains an example to illustrate Theorem \ref{thm:bounded}. Section \ref{sec:Preliminary lemmas} contains three preliminary lemmas. Section \ref{sec:Proof of Theorem thm:bounded} contains the proof of Theorem \ref{thm:bounded}. Section \ref{sec:application} gives an application of Theorem \ref{thm:bounded}.

\section{Example}
\label{sec:Example}

Consider the $C^{1,1}_\mathrm{loc}$ lower bounded semi-algebraic function $f:\mathbb{R}^2 \rightarrow \mathbb{R}$ defined by $f(x,y) = (xy-1)^2$. Even though $f$ is not coercive, the degenerate system 
\begin{equation*}
    \left\{
    \begin{array}{l}
    \gamma \dot{x} + 2 y (xy-1) = 0, \\
    \gamma \dot{y} + 2 x (xy-1) = 0,
    \end{array}
    \right.
     ~~~ x(0) = x_0, ~~~ y(0)=y_0,
\end{equation*}
has bounded solutions. Indeed, $x^2-y^2$ is constant and $(xy-1)^2$ is decreasing, hence $x^4+y^4 = (x^2-y^2)^2+2x^2y^2$ is bounded. By Theorem \ref{thm:bounded}, the perturbed system
\begin{equation*}
    \left\{
    \begin{array}{l}
    \epsilon \ddot{x}_\epsilon + \gamma \dot{x}_\epsilon + 2 y_\epsilon (x_\epsilon y_\epsilon-1) = 0, ~~~x_\epsilon(0) = x_0, ~~~ y_\epsilon(0)=y_0,\\ 
    \epsilon \ddot{y}_\epsilon +\gamma \dot{y}_\epsilon + 2 x_\epsilon (x_\epsilon y_\epsilon-1) = 0, ~~~ \dot{x}_\epsilon(0) = \dot{x}_0, ~~~ \dot{y}_\epsilon(0)=\dot{y}_0, \\
    \end{array}
    \right.
\end{equation*}
has uniformly bounded solutions for all sufficiently small $\epsilon>0$, which must converge to a critical point of $f$ by \cite[Theorem 4.1]{begout2015damped}. 

We next show that the limiting critical point need not agree with the limiting critical point of the degenerate system, establishing that $\epsilon \searrow 0$ is a singular perturbation. In order to do so, let $(x_0,y_0):=(a,-a)$ and $(\dot{x}_0,\dot{y}_0):=(b,b)$ where $a,b>0$. The trajectory of the degenerate system converges to the origin, while the trajectory of the perturbed system converges to a point for which $xy=1$. A numerical illustration is provided in Figure \ref{fig:sinpert}, followed by a proof. It is worth emphasizing that, without Theorem \ref{thm:bounded}, nothing seems to be known about the asymptotic behavior of heavy ball dynamics on this simple example.

\begin{figure}[htb]
    \centering
    \includegraphics[width=.9\textwidth]{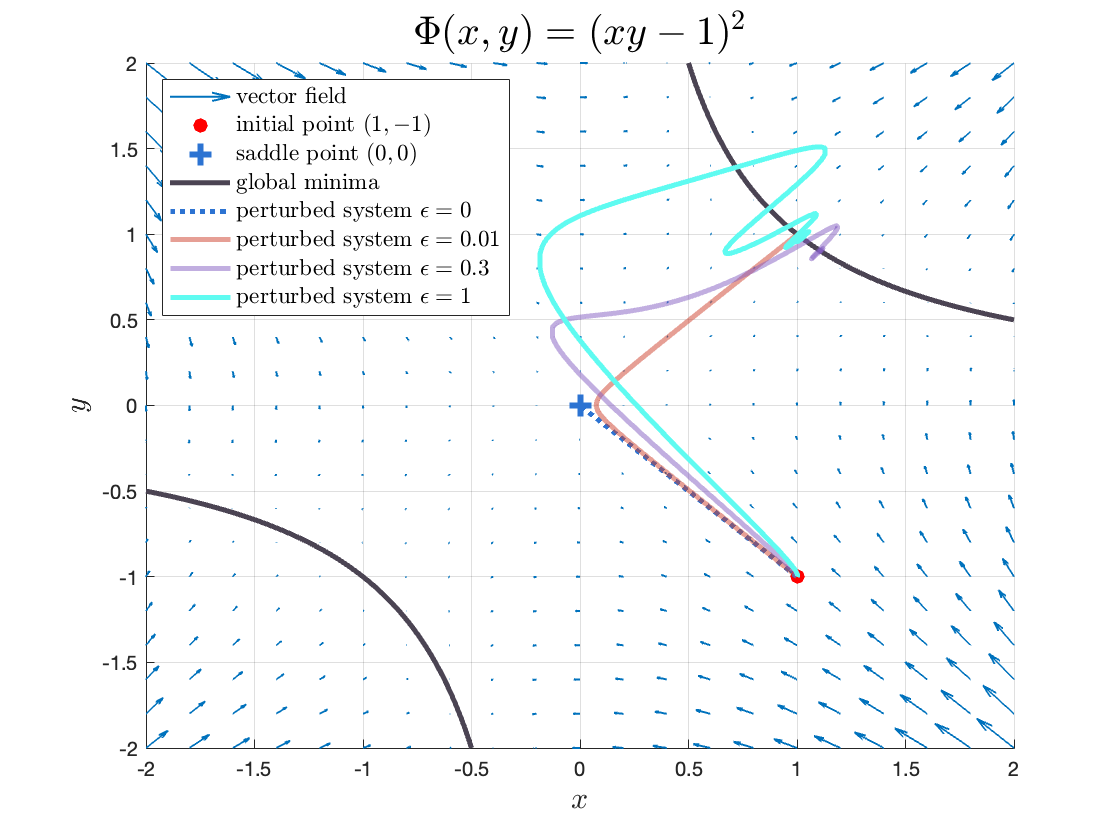}
    \caption{Singular perturbation, $\gamma=0.5$, $\dot{x}_0=\dot{y}_0=0.1$}
    \label{fig:sinpert}
\end{figure}

Since $x^2-y^2$ is constant along trajectories of the degenerate system, the initial condition $(x_0,y_0)=(a,-a)$ implies that $x+y = 0$. This yields the decoupled system
\begin{equation*}
    \left\{
    \begin{array}{l}
    \gamma \dot{x} + 2x(x^2+1) = 0, ~~~ x(0) = a, \\
    \gamma \dot{y} + 2y(y^2+1) = 0, ~~~ y(0)=-a.
    \end{array}
    \right. 
\end{equation*}
It admits the explicit solution
\begin{equation*}
    x(t) = \frac{c}{\sqrt{e^{4t/\gamma}-c^2}}, \quad y(t) = -\frac{c}{\sqrt{e^{4t/\gamma}-c^2}}, \quad c:=\frac{a}{\sqrt{1+a^2}},
\end{equation*}
which converges to $(0,0)$. 

As for the perturbed system, we will show that for all $\epsilon\in(0,\gamma^2/(8a^2+8))$, there exists $t_\epsilon \geqslant 0$ such that $x_\epsilon(t_\epsilon)y_\epsilon(t_\epsilon)= 1/2$. Since the solution $(x_\epsilon,y_\epsilon)$ of the perturbed system is uniformly bounded for all sufficient small $\epsilon>0$, so is its derivative $(\dot{x}_\epsilon,\dot{y}_\epsilon)$. This is a consequence of Lemma \ref{lemma:speed}, which is used to prove Theorem \ref{thm:bounded}. Hence, for all $\epsilon>0$ small enough, the Lyapunov function of the perturbed system evaluated at time $t_\epsilon$ satisfies
\begin{equation*}
    F(t_\epsilon)=(x_\epsilon(t_\epsilon)y_\epsilon(t_\epsilon)-1)^2+\frac{\epsilon}{2}(\dot{x}_\epsilon(t_\epsilon)^2+\dot{y}_\epsilon(t_\epsilon)^2) \leqslant \frac{1}{4}+\frac{1}{4} = \frac{1}{2}.
\end{equation*}

Recall that $F$ is decreasing over $\mathbb{R}_+:=[0,\infty)$ by \cite[Theorem 3.1(ii)]{attouch2000heavy}. As a result, if $(x_\epsilon(t),y_\epsilon(t))\to (0,0)$ as $t\to \infty$, then we obtain the contradiction $1/2 \geqslant F(t_\epsilon) \geqslant \lim_{t\to \infty} F(t) \geqslant 1$. Since $(x_\epsilon(t),y_\epsilon(t))$ converges to one of the critical points of $f$, whose set is given by $\{(0,0)\} \cup \{(x,y) \in \mathbb{R}^2:xy=1\}$, the limit satisfies $xy=1$. 

In the remainder of this section, we reason by contradiction and assume that  there exists $\epsilon\in(0,\gamma^2/(8a^2+8))$ such that for all $t \geqslant 0$, we have $x_\epsilon(t)y_\epsilon(t) \neq 1/2$. Since $x(0)y(0) = -a^2 < 0$, this implies that $x_\epsilon(t)y_\epsilon(t) < 1/2$ for all $t \geqslant 0$. By introducing the new variables $u_\epsilon=x_\epsilon+y_\epsilon$ and $v_\epsilon=x_\epsilon-y_\epsilon$, the perturbed system yields that
\begin{equation*}
    \left\{
    \begin{array}{l}
    \epsilon \ddot{u}_\epsilon + \gamma \dot{u}_\epsilon + 2 u_\epsilon (x_\epsilon y_\epsilon-1) = 0, ~~~u_\epsilon(0) = 0, ~~~ \dot{u}_\epsilon(0) = 2b,\\ %
    \epsilon \ddot{v}_\epsilon +\gamma \dot{v}_\epsilon - 2 v_\epsilon (x_\epsilon y_\epsilon-1) = 0, ~~~ v_\epsilon(0)=2a, ~~~ \dot{v}_\epsilon(0)=0. \\
    \end{array}
    \right.
\end{equation*}

Since $\dot{u}_\epsilon(0)=2b>0$ and $u_\epsilon(0)=0$, by continuity, one has $u_\epsilon(t)>0$ and $\dot{u}_\epsilon(t)>0$ for all $t\in(0,T_1)$, where $T_1:=\inf\{t\in\mathbb{R}_+:\dot{u}(t)=0\}$. Similarly, since $v_\epsilon(0)=2a>0$, by continuity $v_\epsilon(t)>0$ for all $t\in(0,T_2)$ where $T_2:=\inf\{t\in\mathbb{R}_+:v_\epsilon(t)=0\}$. We are going to prove the following claims: 
\begin{enumerate}
    \item $u_\epsilon(t)\geqslant c_1(e^{r_1t}-e^{r_2t})$ for all $t\in[0,T_1)$, where 
    \begin{equation*}
        c_1:=\frac{2b\epsilon}{\sqrt{\gamma^2+4\epsilon}},\quad r_1 := \frac{2}{\sqrt{\gamma^2+4\epsilon}+\gamma},\quad r_2 := -\frac{2}{\sqrt{\gamma^2+4\epsilon}-\gamma}, 
    \end{equation*}
    \item $T_1=\infty$, 
    \item $\dot{v}_\epsilon(t)<0$ for all $t\in[0,T_2)$, 
    \item $v_\epsilon(t)\leqslant (2a+c_2)e^{r_3t}-c_2e^{r_4t}$ for all $t\in[0,T_2)$, where 
    \begin{equation*}
        c_2:=\frac{a\gamma}{\sqrt{\gamma^2-4\epsilon}}-a,\quad r_3 := -\frac{2}{\gamma+\sqrt{\gamma^2-4\epsilon}},\quad r_4 := -\frac{2}{\gamma-\sqrt{\gamma^2-4\epsilon}}, 
    \end{equation*}
    \item $v_\epsilon(t)\geqslant (2a+c_3)e^{r_5t}-c_3e^{r_6t}$ for all $t\in[0,T_2)$, where 
    \begin{gather*}
        c_3:=\frac{a\gamma}{\sqrt{\gamma^2-8(a^2+1)\epsilon}}-a, \\
        r_5 := -\frac{4(a^2+1)}{\gamma+\sqrt{\gamma^2-8(a^2+1)\epsilon}},\quad r_6 := -\frac{4(a^2+1)}{\gamma-\sqrt{\gamma^2-8(a^2+1)\epsilon}}, 
    \end{gather*}
    \item $T_2=\infty$. 
\end{enumerate}
Together, these claims imply that $u_\epsilon(t)\to\infty$ and $v_\epsilon(t)\to 0$ as $t\to\infty$. This yields the contradiction $1/2 > x_\epsilon(t)y_\epsilon(t)=(u_\epsilon(t)^2-v_\epsilon(t)^2)/4\to\infty$.

\textit{Proof of 1.} Since $u_\epsilon(t)>0$ and $x_\epsilon(t)y_\epsilon(t)-1<-1/2$ for all $t\in(0,T_1)$, one has the following differential inequality
\begin{equation}
\label{eq:DI}
    \epsilon \ddot{u}_\epsilon(t) + \gamma \dot{u}_\epsilon(t) - u_\epsilon(t) \geqslant 0, \quad \forall t\in [0,T_1). 
\end{equation}
The corresponding differential equation
\begin{equation*}
    \epsilon \ddot{\bar{u}}_\epsilon(t) + \gamma \dot{\bar{u}}_\epsilon(t) - \bar{u}_\epsilon(t) = 0, ~~~ \bar{u}_\epsilon(0) = 0, ~~~ \dot{\bar{u}}_\epsilon(0) = 2b, \quad \forall t\in [0,T_1),
\end{equation*}
admits the unique solution
\begin{equation*}
    \bar{u}_\epsilon(t) = c_1(e^{r_1t}-e^{r_2t}), \quad \forall t\in [0,T_1). 
\end{equation*}
Let $t_1\in(0,T_1)$. Since $\bar{u}_\epsilon(t_1)>0$, by the comparison theorem \cite[Theorem 2]{popa2002differential}, we have $u_\epsilon(t)\geqslant \bar{u}_\epsilon(t)$ for all $t \in [t_1,T_1)$. Since $u_\epsilon(0)\geqslant \bar{u}_\epsilon(0)$ and $t_1$ is arbitrary in $(0,T_1)$, we actually have $u_\epsilon(t)\geqslant \bar{u}_\epsilon(t)$ for all $t \in [0,T_1)$.

\textit{Proof of 2.} Assume, for the sake of contradiction, that $T_1<\infty$. By continuity, $u_\epsilon(T_1)\geqslant \bar{u}_\epsilon(T_1)>0$, $\dot{u}_\epsilon(T_1)=0$, and $\ddot{u}_\epsilon(T_1)>0$ (by \eqref{eq:DI}). Thus, by continuity, there exists $\delta\in(0,T_1)$ such that $\ddot{u}_\epsilon(t)>0$ for all $t\in[T_1-\delta,T_1]$. This shows that $\dot{u}_\epsilon$ is strictly increasing over $[T_1-\delta,T_1]$. Note that $\dot{u}_\epsilon(T_1-\delta)>0$, thus by monotonicity $0 = \dot{u}_\epsilon(T_1)>\dot{u}_\epsilon(T_1-\delta)>0$, a contradiction. We conclude that $T_1=\infty$. This establishes the second claim.

\textit{Proof of 3.} Since $x_\epsilon(0)y_\epsilon(0)-1<-1/2$, $v_\epsilon(0)>0$ and $\dot{v}_\epsilon(0)=0$, one has $\ddot{v}_\epsilon(0)<0$. By continuity, one has $\dot{v}_\epsilon(t)<0$ for all $t\in(0,T_3)$ where $T_3:=\inf\{t>0:\dot{v}_\epsilon(t)=0\}>0$. It suffices to show $T_3\geqslant T_2$. Assume $T_3<T_2$ for the sake of contradiction. Since $\dot{v}_\epsilon(T_3)=0$, $v_\epsilon(T_3)>0$ and $x_\epsilon(T_3)y_\epsilon(T_3)-1<-1/2$, one can conclude that $\ddot{v}_\epsilon(T_3)<0$. By continuity, there exists $\delta\in(0,T_3)$ such that $\ddot{v}_\epsilon(t)<0$ for all $t\in[T_3-\delta,T_3]$, i.e., $\dot{v}_\epsilon$ is strictly decreasing over $[T_3-\delta,T_3]$. This yields the contradiction $0>\dot{v}_\epsilon(T_3-\delta)>\dot{v}_\epsilon(T_3) = 0$. Hence $T_3\geqslant T_2$ and $\dot{v}_\epsilon(t)<0$ for all $t\in[0,T_2)$. 

\textit{Proof of 4.} Since $v_\epsilon(t)>0$ and $x_\epsilon(t)y_\epsilon(t)-1<-1/2$ for all $t\in(0,T_2)$, one has the following differential inequality
\begin{equation*}
    \epsilon \ddot{v}_\epsilon(t) + \gamma \dot{v}_\epsilon(t) + v_\epsilon(t) \leqslant 0, \quad \forall t\in [0,T_2). 
\end{equation*}
The corresponding differential equation
\begin{equation*}
    \epsilon \ddot{\bar{v}}_\epsilon(t) + \gamma \dot{\bar{v}}_\epsilon(t) + \bar{v}_\epsilon(t) = 0, ~~~ \bar{v}_\epsilon(0) = 2a, ~~~ \dot{\bar{v}}_\epsilon(0) = 0, \quad \forall t\in [0,T_2). 
\end{equation*}
admits the unique solution
\begin{equation*}
    \bar{v}_\epsilon(t) = (2a+c_2)e^{r_3t}-c_2e^{r_4t}, \quad \forall t\in [0,T_2). 
\end{equation*}
Since $\bar{v}_\epsilon(t)>0$ for any $t\in(0,T_2)$ and $v_\epsilon(0)\leqslant\bar{v}_\epsilon(0)$, the comparison theorem \cite[Theorem 2]{popa2002differential} yields $v_\epsilon(t)\leqslant\bar{v}_\epsilon(t)$ for all $t\in[0,T_2)$.

\textit{Proof of 5.} Since $u_\epsilon$ is increasing over $\mathbb{R}_+$ and $v_\epsilon$ is decreasing over $[0,T_2)$, one notices that $x_\epsilon y_\epsilon=(u_\epsilon^2-v_\epsilon^2)/4$ is increasing over $[0,T_2)$. This shows that $x_\epsilon(t) y_\epsilon(t)\geqslant x_\epsilon(0) y_\epsilon(0)=-a^2$. Combined with the fact that $v_\epsilon(t)>0$ for all $t \in (0,T_2)$, one can obtain the following differential inequality
\begin{equation*}
    \epsilon \ddot{v}_\epsilon(t) + \gamma \dot{v}_\epsilon(t) + 2(a^2+1)v_\epsilon(t) \geqslant 0, \quad \forall t\in [0,T_2). 
\end{equation*}
The corresponding differential equation
\begin{equation*}
    \epsilon \ddot{\tilde{v}}_\epsilon(t) + \gamma \dot{\tilde{v}}_\epsilon(t) + 2(a^2+1)\tilde{v}_\epsilon(t) = 0, ~~~ \tilde{v}_\epsilon(0) = 2a, ~~~ \dot{\tilde{v}}_\epsilon(0) = 0, \quad \forall t\in [0,T_2), 
\end{equation*}
admits the unique solution
\begin{equation*}
    \tilde{v}_\epsilon(t) = (2a+c_3)e^{r_5t}-c_3e^{r_6t}, \quad \forall t\in [0,T_2). 
\end{equation*}
Since $\tilde{v}_\epsilon(t)>0$ for any $t\in(0,T_2)$ and $v_\epsilon(0)\geqslant\tilde{v}_\epsilon(0)$, the comparison theorem \cite[Theorem 2]{popa2002differential} yields $v_\epsilon(t)\geqslant\tilde{v}_\epsilon(t)$ for all $t\in[0,T_2)$.

\textit{Proof of 6.} Assume, for the sake of contradiction, that $T_2<\infty$. Then by continuity, one has $v_\epsilon(T_2)\geqslant\tilde{v}_\epsilon(T_2)$. Notice that 
\begin{equation*}
    \dot{\tilde{v}}_\epsilon(t) = (2a+c_3)r_5e^{r_5t}-c_3r_6e^{r_6t}<0, \quad \forall t\in [0,T_2). 
\end{equation*}
Indeed, the initial condition $\dot{\tilde{v}}_\epsilon(0)=0$ implies that $(2a+c_3)r_5=c_3r_6<0$ and $r_5>r_6$. Since $\tilde{v}_\epsilon(t)\to 0$ as $t\to\infty$, by monotonicity, one has $\tilde{v}_\epsilon(T_2)>0$. This yields the contradiction $0 = v_\epsilon(T_2) \geqslant \tilde{v}_\epsilon(T_2)>0$.

\section{Preliminary lemmas}
\label{sec:Preliminary lemmas}

Let $B(a,r)$ and $\mathring{B}(a,r)$ respectively denote the closed and open balls of center $a \in \mathbb{R}^n$ and radius $a\geqslant 0$. Lemma \ref{lemma:tracking} extends \cite[Theorem 5.1]{attouch2000heavy} of Attouch \textit{et al.} by showing that $\epsilon\searrow 0$ is not only a regular perturbation in finite time for fixed initial conditions, but also for a compact set of initial conditions. 

\begin{lemma}\label{lemma:tracking}
    Let $f:\mathbb{R}^n\rightarrow\mathbb{R}$ be a $C^{1,1}_\mathrm{loc}$ lower bounded function, $X_0$ be a compact subset of $\mathbb{R}^n$, and $r_0,T \geqslant 0$. For all $\delta>0$, there exists $\bar{\epsilon}>0$ such that for all $\epsilon\in (0,\bar{\epsilon}]$ and for any global solution $x_\epsilon$ to \eqref{eq:P_eps} initialized in $X_0 \times B(0,r_0)$, there exists a global solution to $x$ to \eqref{eq:P_0} initialized in $X_0$ such that $\|x_\epsilon(t) - x(t)\| \leqslant \delta$ for all $t\in[0,T]$. 
\end{lemma}
\begin{proof}
The set of solutions $\{x_\epsilon\}_{\epsilon \in (0,1]}$ to \eqref{eq:P_eps} up to time $T$ initialized in $X_0 \times B(0,r_0)$ is precompact w.r.t. the topology of uniform convergence. Indeed, since
    \begin{equation*}
        \frac{d}{dt}\left(f(x_\epsilon)+\frac{\epsilon}{2}\|\dot{x}_\epsilon\|^2\right) = -\gamma\|\dot{x}_\epsilon\|^2
    \end{equation*}
   (as observed in \cite{haraux1986asymptotics,alvarez2000minimizing,attouch2000heavy} when $\epsilon = 1$) we have
    \begin{subequations}
    \label{eq:L2}
        \begin{align}
        \int_0^T \|\dot{x}_\epsilon(\tau)\|^2d\tau & = \frac{1}{\gamma} \left( f(x_\epsilon(0)) - f(x_\epsilon(T)) + \frac{\epsilon}{2}\|\dot{x}_\epsilon(0)\|^2 - \frac{\epsilon}{2}\|\dot{x}_\epsilon(T)\|^2\right) \\
        & \leqslant \frac{1}{\gamma}\left(\sup_{X_0}f - \inf_{\mathbb{R}^n} f+\frac{\epsilon r_0^2}{2}\right).
        \end{align}
    \end{subequations}
    $\{x_\epsilon\}_{\epsilon\in(0,1]}$ is equicontinuous because for all $0\leqslant s \leqslant t \leqslant T$ we have
        \begin{align*}
            \|x_\epsilon(s)-x_\epsilon(t)\| & \leqslant \int_s^t\|\dot{x}_\epsilon(\tau)\|d\tau \\
            & \leqslant \sqrt{t-s} \int_s^t\|\dot{x}_\epsilon(\tau)\|^2d\tau \\
            & \leqslant \sqrt{t-s} \sqrt{\frac{1}{\gamma}\left(\sup_{X_0}f - \inf_{\mathbb{R}^n} f + \frac{r_0^2}{2}\right)}
        \end{align*}
    and $\{x_\epsilon(t)\}_{\epsilon\in(0,1]}$ is bounded for all $t\in [0,T]$ because
    \begin{subequations}
    \begin{align}
        \|x_\epsilon(t)\| &\leqslant \|x_\epsilon(0)\| + \int_0^T\|\dot{x}_\epsilon(\tau)\|d\tau \\
        &\leqslant \sup_{x\in X_0}\|x\| + \sqrt{T}\sqrt{\frac{1}{\gamma}\left(\sup_{X_0}f - \inf_{\mathbb{R}^n} f  + \frac{r_0^2}{2}\right)}. \label{eq:x(t)} 
    \end{align}
    \end{subequations}
    
    We next show that for any sequence $\epsilon_k \searrow 0$, there exists a subsequence (again denoted $\epsilon_k$) such that $x_{\epsilon_k}$ converges uniformly on $[0,T]$ to a solution of \eqref{eq:P_0} initialized in $X_0$. The conclusion of the lemma then readily follows. Assume for the sake of contradiction that there exists $\delta>0$ such that, for all $\bar{\epsilon}>0$, there exists $\epsilon \in (0,\bar{\epsilon}]$ and a global solution $x_\epsilon$ to \eqref{eq:P_eps} initialized in $X_0 \times B(0,r_0)$ such that, for all global solution to $x$ to \eqref{eq:P_0} initialized in $X_0$, it holds that $\|x_\epsilon(t) - x(t)\| > \delta$ for some $t\in[0,T]$. We can then generate a sequence $\epsilon_k \searrow 0$ such that, for all global solution to $x$ to \eqref{eq:P_0} initialized in $X_0$, it holds that $\|x_{\epsilon_k}(t) - x(t)\| > \delta$ for some $t\in[0,T]$. Since there exists a subsequence (again denoted $\epsilon_k$) such that $x_{\epsilon_k}$ converges on $[0,T]$ uniformly to a solution of \eqref{eq:P_0}, we obtain a contradiction.
    
    Consider a sequence $\epsilon_k \searrow 0$. By the Arzelà-Ascoli theorem \cite[Theorem 1 p. 13]{aubin1984differential}, there exists a subsequence (again denoted $\epsilon_k$) such that $x_{\epsilon_k}$ converges uniformly on $[0,T]$ to a continuous function $u$. Recall that $(\dot{x}_{\epsilon_k})_{k\in \mathbb{N}}$ is bounded in $L^2([0,T],\mathbb{R}^n)$ due to \eqref{eq:L2}. By further taking a subsequence, $\dot{x}_{\epsilon_k}$ thus converges weakly to a function $v$ in $L^2([0,T],\mathbb{R}^n)$ \cite[Theorem 17, p. 283]{Fitzpatrick2010}. 
    
    Naturally, $\dot{u}=v$ almost everywhere on $(0,T)$. Indeed, since $x_{\epsilon_k}$ is absolutely continuous, for all $s,t\in [0,T]$ we have $x_{\epsilon_k}(t) - x_{\epsilon_k}(s) = \int_s^t \dot{x}_{\epsilon_k}(\tau)d\tau$
    and taking the limit yields $u(t) - u(s) = \int_s^t v(\tau)d\tau$.
    In addition, $\ddot{x}_{\epsilon_k}$ converges to $\dot{v}$ as a distribution. Indeed, for any test function $\varphi: (0,T) \rightarrow \mathbb{R}^n$ (i.e., infinitely differentiable with compact support), we have
    \begin{align*}
        \int_0^T \langle \ddot{x}_{\epsilon_k}(\tau),\varphi(\tau) \rangle d\tau & = \int_0^T \langle \dot{x}_{\epsilon_k}(\tau),\dot{\varphi}(\tau) \rangle d\tau \\
        & \rightarrow \int_0^T \langle v(\tau),\dot{\varphi}(\tau) \rangle d\tau \\
        & = \int_0^T \langle \dot{v}(\tau),\varphi(\tau) \rangle d\tau.
    \end{align*}
Passing to the limit in $\epsilon_k\ddot{x}_{\epsilon_k} + \gamma\dot{x}_{\epsilon_k} + \nabla f(x_{\epsilon_k}) = 0$ yields 
$\gamma \dot{u} + \nabla f(u) = 0$ in the distribution sense. Since $\gamma\dot{u} + \nabla f(u) \in L^2([0,T],\mathbb{R}^n) \subset L^1_\mathrm{loc}([0,T],\mathbb{R}^n)$, it holds that $\gamma \dot{u} + \nabla f(u) = 0$ almost everywhere on $(0,T)$. As a result, $u(t) - u(0) = \int_0^t \dot{u}(\tau)\;d\tau = -\int_0^t \nabla f(u(\tau))/\gamma d\tau$. As the integral of a continuous function, $u$ is $C^1$ on $(0,T)$ \cite[Theorem 6.20]{rudin1964principles}. Hence $\gamma \dot{u} + \nabla f(u) = 0$ everywhere on $(0,T)$. Since $f\in C^{1,1}_\mathrm{loc}$, by the Picard–Lindel{\"o}f theorem \cite[Theorem 3.1 p. 12]{coddington1955theory}, $u$ is a solution to \eqref{eq:P_0} on $[0,T]$.
\end{proof}

Attouch \textit{et al.} \cite[Theorem 3.1 (ii)]{attouch2000heavy} show that the velocities of heavy ball trajectories remain bounded throughout time. Lemma \ref{lemma:speed} provides conditions ensuring that they are uniformly bounded with respect to the mass $\epsilon$.

\begin{lemma}\label{lemma:speed}
    Let $f:\mathbb{R}^n\rightarrow\mathbb{R}$ be a $C^{1,1}_\mathrm{loc}$ lower bounded function, $X \subset \mathbb{R}^n$ be bounded, and $\gamma,r_0>0$. There exists $r>0$ such that for all $\epsilon,T>0$, if $x_\epsilon:[0,T]\rightarrow X$ is a solution to \eqref{eq:P_eps} 
    such that $\|\dot{x}_\epsilon(0)\|\leqslant r_0$, then $\|\dot{x}_\epsilon(t)\|\leqslant r$ for all $t\in[0,T]$.
\end{lemma}
\begin{proof}

    Let $t \in [0,T)$ and $h \in(0,T-t)$. Due to \eqref{eq:P_eps}, the difference quotients 
    \begin{equation*}
        u_{\epsilon,h}(t) := \frac{\dot{x}_\epsilon(t+h)-\dot{x}_\epsilon(t)}{h}, \quad v_{\epsilon,h}(t):=\frac{\nabla f(x_\epsilon(t+h))-\nabla f(x_\epsilon(t))}{h}. 
    \end{equation*}
    satisfy $\epsilon \dot{u}_{\epsilon,h}(t)+\gamma u_{\epsilon,h}(t) + v_{\epsilon,h}(t) = 0$. Following \cite[Equations (5.10)-(5.13)]{attouch2000heavy}, we take the inner product with $\epsilon u_{\epsilon,h}$ and obtain
    \begin{align*}
        \epsilon^2 \langle \dot{u}_{\epsilon,h}(t) , u_{\epsilon,h}(t) \rangle + \epsilon \gamma \| u_{\epsilon,h}(t) \|^2  &= \epsilon \langle -v_{\epsilon,h}(t) , u_{\epsilon,h}(t) \rangle \\
        & \leqslant \epsilon \|\gamma^{-1/2} v_{\epsilon,h}(t)\| \|\gamma^{1/2} u_{\epsilon,h}(t)\| \\
        & \leqslant \frac{\epsilon}{2\gamma} \|v_{\epsilon,h}(t)\|^2 + \frac{\epsilon \gamma}{2} \|u_{\epsilon,h}(t)\|^2 \\
        & = \frac{\epsilon}{2\gamma} \left\|\frac{\nabla f(x_\epsilon(t+h))-\nabla f(x_\epsilon(t))}{h}\right\|^2 + \frac{\epsilon \gamma}{2} \|u_{\epsilon,h}(t)\|^2 \\
        & \leqslant \frac{\epsilon L^2}{2\gamma} \left\|\frac{x_\epsilon(t+h)-x_\epsilon(t)}{h}\right\|^2 + \frac{\epsilon \gamma}{2} \|u_{\epsilon,h}(t)\|^2
    \end{align*}
    where $L$ is a Lipschitz constant of the gradient of $f$ on $X$.
    Integrating yields
    \begin{equation}
    \label{eq:h}
        \|\epsilon u_{\epsilon,h}(t)\|^2 \leqslant \frac{\epsilon L^2}{\gamma} \int_0^t \left\|\frac{x_\epsilon(\tau+h)-x_\epsilon(\tau)}{h}\right\|^2 d\tau + \|\epsilon u_{\epsilon,h}(0)\|^2 . 
    \end{equation}
    By the mean value theorem and \eqref{eq:L2}, we have
    \begin{equation}
    \label{eq:eps_bound}
        \left\|\frac{x_\epsilon(t+h)-x_\epsilon(t)}{h}\right\| \leqslant \sup_{[t,t+h]} \|\dot{x}_\epsilon\| \leqslant \sqrt{\frac{2}{\epsilon}\left(\sup_X f -\inf_{\mathbb{R}^n} f\right)+r_0^2}.
    \end{equation}
    By the dominated convergence theorem, taking the limit as $h\rightarrow 0$ in \eqref{eq:h} yields
    \begin{align*}
        \|\epsilon \ddot{x}_\epsilon(t)\|^2 & \leqslant \frac{\epsilon L^2}{\gamma}\int_0^t\|\dot{x}_\epsilon(\tau)\|^2d\tau + \|\epsilon \ddot{x}_\epsilon(0)\|^2 \\
        & \leqslant \frac{\epsilon L^2}{\gamma}\int_0^t\|\dot{x}_\epsilon(\tau)\|^2d\tau + (\|\gamma \dot{x}_\epsilon(0)\| + \|\nabla f(x_\epsilon(0))\|)^2.
    \end{align*}
    We conclude that 
    \begin{align*}
        \|\dot{x}_\epsilon(t)\| & = \frac{1}{\gamma} \| \epsilon \ddot{x}_\epsilon(t) + \nabla f(x_\epsilon(t))\| \\
        & \leqslant \frac{1}{\gamma} (\| \epsilon \ddot{x}_\epsilon(t) \| +  \| \nabla f(x_\epsilon(t))\|) \\
        & \leqslant \frac{1}{\gamma} \left[ \sqrt{\frac{\epsilon L^2}{\gamma}\int_0^t\|\dot{x}_\epsilon(\tau)\|^2d\tau + (\|\gamma \dot{x}_\epsilon(0)\| + \|\nabla f(x_\epsilon(0))\|)^2} + \| \nabla f(x_\epsilon(t))\| \right] \\
        & \leqslant \frac{1}{\gamma} \left[ \sqrt{\frac{\epsilon L^2}{\gamma^2}\left(\sup_{X}f - \inf_{\mathbb{R}^n} f+\frac{\epsilon r_0^2}{2}\right) + (\gamma r_0 + \sup_X \|\nabla f\|)^2} + \sup_X \|\nabla f\| \right]. \\
    \end{align*}
    The above upper bound is increasing with $\epsilon$, while the upper bound in \eqref{eq:eps_bound} decreases with $\epsilon$. Taking the minimum of the two and maximizing over $\epsilon>0$ yields a bound $r$ that is independent of $\epsilon$. Since $\dot{x}$ is continuous, the bound also holds at time $t = T$.
\end{proof}

The Kurdyka-\L{}ojasiewicz inequality enables one to relate the length of gradient trajectories with the function variation \cite[Theorem 2 b)]{kurdyka1998gradients} (see also \cite[Proposition 7]{josz2023global}). Lemma \ref{lemma:length} provides such a length formula for heavy ball dynamics.

\begin{lemma}
\label{lemma:length}
Let $f:\mathbb{R}^n\rightarrow\mathbb{R}$ be a $C^{1,1}_\mathrm{loc}$ lower bounded definable function, $X$ be a bounded subset of $\mathbb{R}^n$, and $\gamma,r,\bar{\epsilon}>0$. There exist $\eta>0$ and a diffeomorphism $\varphi:\mathbb{R}_+\rightarrow \mathbb{R}_+$ such that, for all $\epsilon \in (0,\bar{\epsilon}]$ and $T\geqslant 0$, if $x_\epsilon:[0,T]\rightarrow X$ is a solution to $(P_\epsilon)$ such that $\|\dot{x}_\epsilon(t)\| \leqslant r$ for all $t\in[0,T]$, then
\begin{equation}\label{eq:length_formula}
    \int_0^T\|\dot{x}_\epsilon(t)\|dt \leqslant \varphi\left(f(x_\epsilon(0))-f(x_\epsilon(T))+\eta\epsilon\right).
\end{equation}
\end{lemma}
\begin{proof}
Let $\epsilon \in (0,\bar{\epsilon}]$, $T\geqslant 0$, and $x_\epsilon:[0,T]\rightarrow X$ be a solution to $(P_\epsilon)$ such that $\|\dot{x}_\epsilon(t)\| \leqslant r$ for all $t\in[0,T]$. The proof is devoted to building $\eta>0$ and a diffeomorphism $\varphi:\mathbb{R}_+\rightarrow \mathbb{R}_+$ satisfying \eqref{eq:length_formula} that are independent of $\epsilon$ and $T$. They will be displayed at the end of the proof in  \eqref{eq:eta_varphi}. We construct $\varphi$ by introducing two parameters $\alpha$ and $\beta$ whose values we will tune throughout the proof in order to obtain the desired properties.

Following Zavriev and Kostyuk \cite{zavriev1993heavy}, consider the Lyapunov function $H_\alpha: \mathbb{R}^n \times \mathbb{R}^n \rightarrow \mathbb{R}$ defined by $H_\alpha(x,y) := f(x) + \alpha \|x-y\|^2$ where $\alpha>0$.  Following Bo{\c{t}} \textit{et al.} \cite[Theorem 3.2]{boct2020second}, consider the auxiliary dynamics $u_\epsilon=x_\epsilon+\beta\dot{x}_\epsilon$ where $\beta>0$. Notice that $H_\alpha(u_\epsilon,x_\epsilon) = f(x_\epsilon+\beta\dot{x}_\epsilon) + \alpha \beta^2\|\dot{x}_\epsilon\|^2$. Let $L \geqslant \max \{1,\bar{\epsilon}\}$ denote a Lipschitz constant of $f$ and $\nabla f$ on $B(X,r) := X+B(0,r)$. Since $u_\epsilon(t)\in B(X,r)$ and $x_\epsilon(t)\in X$ for all $t\in[0,T]$, we have
\begin{align*}
    \frac{d}{dt}H_\alpha(u_\epsilon,x_\epsilon) = & \langle \nabla f(x_\epsilon+\beta\dot{x}_\epsilon), \dot{x}_\epsilon+\beta\ddot{x}_\epsilon\rangle + 2\alpha\beta^2\langle \dot{x}_\epsilon,\ddot{x}_\epsilon\rangle \\
    = & \langle \nabla f(x_\epsilon+\beta\dot{x}_\epsilon)-\nabla f(x_\epsilon), \dot{x}_\epsilon+\beta\ddot{x}_\epsilon\rangle + \langle \nabla f(x_\epsilon), \dot{x}_\epsilon+\beta\ddot{x}_\epsilon\rangle 
    \\[2mm]
    &+ 2\alpha\beta^2\langle \dot{x}_\epsilon,\ddot{x}_\epsilon\rangle \\[2mm]
    \leqslant & L\beta\|\dot{x}_\epsilon\|\|\dot{x}_\epsilon+\beta\ddot{x}_\epsilon\|
    -\langle \gamma \dot{x}_\epsilon + \epsilon\ddot{x}_\epsilon, \dot{x}_\epsilon+\beta\ddot{x}_\epsilon\rangle + 2\alpha\beta^2\langle \dot{x}_\epsilon,\ddot{x}_\epsilon\rangle \\[2mm]
    \leqslant & L\beta(\|\dot{x}_\epsilon\|^2 + \beta\|\dot{x}_\epsilon\|\|\ddot{x}_\epsilon\|)
    - \gamma \|\dot{x}_\epsilon\|^2 - (\gamma \beta+\epsilon) \langle \dot{x}_\epsilon , \ddot{x}_\epsilon \rangle
    \\[2mm]
    & - \epsilon \beta \|\ddot{x}_\epsilon\|^2 + 2\alpha\beta^2\langle \dot{x}_\epsilon,\ddot{x}_\epsilon\rangle \\[2mm]
    \leqslant & L\beta(\|\dot{x}_\epsilon\|^2 + \beta\|\dot{x}_\epsilon\|^2/2+\beta\|\ddot{x}_\epsilon\|^2/2) - \gamma\|\dot{x}_\epsilon\|^2 - \epsilon\beta\|\ddot{x}_\epsilon\|^2 \\[2mm]
    & + \underbrace{(2\alpha\beta^2-\gamma\beta-\epsilon)}_{=~0}\langle \dot{x}_\epsilon,\ddot{x}_\epsilon\rangle \\
    = & -a\|\dot{x}_\epsilon\|^2-b\|\ddot{x}_\epsilon\|^2
\end{align*}
where $a:= \gamma-L\beta(1+\beta/2)>0$ and $b = \beta(\epsilon-L\beta/2)>0$ by taking $\beta < \min \{\sqrt{1+2\gamma/L}-1 , 2\epsilon/L\}$. The zero term is obtained by taking $\alpha = (\gamma\beta+\epsilon)/(2\beta^2)$, which is greater than or equal to $1/4$. Indeed, $\beta \leqslant \sqrt{1+2\gamma/L}-1 \leqslant \sqrt{1+2\gamma}-1 \leqslant 2\gamma$ since $L\geqslant 1$. It then suffices to see that 
\begin{equation*}
    \alpha\geqslant \frac{1}{4} \iff \frac{\epsilon+\gamma\beta}{2\beta^2}\geqslant \frac{1}{4} \iff (\beta-\gamma)^2 \leqslant 2\epsilon+\gamma^2 \impliedby \beta \in (0,2\gamma].
\end{equation*}
Furthermore, we have
\begin{align*}
    \|\nabla H_\alpha(u_\epsilon,x_\epsilon)\| & \leqslant \|\nabla f(u_\epsilon)\|+2\alpha\|u_\epsilon-x_\epsilon\| \\
    &\leqslant \|\nabla f(u_\epsilon)- \nabla f(x_\epsilon)\|+\|\nabla f(x_\epsilon)\|+2\alpha\beta^2\|\dot{x}_\epsilon\| \\
    &\leqslant L\beta\|\dot{x}_\epsilon\| + \|\epsilon\ddot{x}_\epsilon+\gamma\dot{x}_\epsilon\|+2\alpha\beta^2\|\dot{x}_\epsilon\| \\
    &\leqslant c\|\dot{x}_\epsilon\| +  \epsilon\|\ddot{x}_\epsilon\|
\end{align*}
where $c:= L\beta+\gamma+\epsilon+\gamma\beta>0$.

We say that $f$ attains a critical value $v \in \mathbb{R}$ in a set $S \subset \mathbb{R}^n$ if there exists $x\in S$ such that $f(x) = v$ and $\nabla f(x) = 0$. By the definable Morse-Sard theorem \cite[Corollary 9]{bolte2007clarke}, $f$ has finitely many critical values in $\mathbb{R}^n$. Let $m \in \mathbb{N}\setminus \{0\}$ be an upper bound on the number of critical values of $f$ in $\overline{X}$, i.e., the closure of $X$. Since $\nabla H_\alpha(x,y) = (\nabla f(x)+2\alpha(x-y) , 2\alpha(y-x))^\top$, the critical values of $f$ in $\overline{X}$ are the same as those of $H_\alpha$ in $\overline{B(X,r)\times X} = B(\overline{X},r)\times \overline{X}$. We let $V$ denote this set of critical values if they exist, otherwise let $V:=\{0\}$. Also, let $d(x,S) := \inf \{ \|x-y\| : y \in S\}$ be the distance of point $x\in \mathbb{R}^n$ to a set $S \subset \mathbb{R}^n$. This enables us to define the function $\widetilde{H}_\alpha(x) := d(H_\alpha(x),V)$ for all $x\in\mathbb{R}^n$. 

Since $\alpha \geqslant 1/4$, by \cite[Proposition 3]{josz2023convergence} there exists a concave definable diffeomorphism $\psi:\mathbb{R}_+ \rightarrow \mathbb{R}_+$ such that $\| \nabla (\psi \circ \widetilde{H}_\alpha)(u_\epsilon,x_\epsilon)\| \geqslant 1$ whenever $0\notin \partial\widetilde{H}_\alpha(u_\epsilon,x_\epsilon)$. In other words, $\psi'(\widetilde{H}_\alpha(u_\epsilon,x_\epsilon)) \geqslant 1/\|\nabla \widetilde{H}_\alpha(u_\epsilon,x_\epsilon)\| = 1/\|\nabla H_\alpha(u_\epsilon,x_\epsilon)\|$. Since $H_\alpha(u_\epsilon,x_\epsilon)$ is decreasing, $\widetilde{H}_\alpha(u_\epsilon,x_\epsilon)$ is either decreasing or increasing. If it is decreasing, then 
\begin{align*}
    \frac{d}{dt} (\psi \circ \widetilde{H}_\alpha)(u_\epsilon,x_\epsilon) & = \psi' (\widetilde{H}_\alpha(u_\epsilon,x_\epsilon)) \frac{d}{dt}\widetilde{H}_\alpha(u_\epsilon,x_\epsilon) \\
    & = \psi' (\widetilde{H}_\alpha(u_\epsilon,x_\epsilon)) \frac{d}{dt}H_\alpha(u_\epsilon,x_\epsilon) \\
    & \leqslant -\frac{a\|\dot{x}_\epsilon\|^2 + b\|\ddot{x}_\epsilon\|^2}{c\|\dot{x}_\epsilon\| +  \epsilon\|\ddot{x}_\epsilon\|}.
\end{align*}
If it is increasing, then
\begin{align*}
    \frac{d}{dt} (\psi \circ \widetilde{H}_\alpha)(u_\epsilon,x_\epsilon) & = \psi' (\widetilde{H}_\alpha(u_\epsilon,x_\epsilon)) \frac{d}{dt}\widetilde{H}_\alpha(u_\epsilon,x_\epsilon) \\
    & = -\psi' (\widetilde{H}_\alpha(u_\epsilon,x_\epsilon)) \frac{d}{dt}H_\alpha(u_\epsilon,x_\epsilon) \\
    & \geqslant \frac{a\|\dot{x}_\epsilon\|^2 + b\|\ddot{x}_\epsilon\|^2}{c\|\dot{x}_\epsilon\| +  \epsilon\|\ddot{x}_\epsilon\|}.
\end{align*}
Together, this yields
\begin{align*}
    \|\dot{x}_\epsilon\| & \leqslant c^{-1} (\|c \dot{x}_\epsilon\| + \|\epsilon \ddot{x}_\epsilon\|) \\
    & \leqslant \frac{2}{c} \frac{\|c \dot{x}_\epsilon\|^2 + \|\epsilon \ddot{x}_\epsilon\|^2}{\|c \dot{x}_\epsilon\| + \|\epsilon \ddot{x}_\epsilon\|} \\
    & = \frac{2}{c} \left( \frac{c^2}{a} \frac{a\|\dot{x}_\epsilon\|^2}{\|c \dot{x}_\epsilon\| + \|\epsilon \ddot{x}_\epsilon\|} + \frac{\epsilon^2}{b} \frac{b\|\ddot{x}_\epsilon\|^2}{\|c \dot{x}_\epsilon\| + \|\epsilon \ddot{x}_\epsilon\|} \right) \\
    & \leqslant \frac{2}{c} \left( \frac{c^2}{a} + \frac{\epsilon^2}{b} \right) \frac{a\|\dot{x}_\epsilon\|^2+b\|\ddot{x}_\epsilon\|^2}{\|c \dot{x}_\epsilon\| + \|\epsilon \ddot{x}_\epsilon\|} \\
    & \leqslant 2\left( \frac{c}{a} + \frac{\epsilon^2}{bc} \right) \left| \frac{d}{dt}(\psi \circ \widetilde{H}_\alpha)(u_\epsilon,x_\epsilon) \right|.
\end{align*}

Consider the times $t \in \{t_1,\hdots,t_k\}$ where $0 < t_1 < \cdots < t_k < T$ and potentially $t\in (t_k,T)$ such that $0\in\partial \widetilde{H}_\alpha(u_\epsilon(t),x_\epsilon(t))$. 
For notational convenience, let $t_0 := 0$ and $t_{k+1}:=T$ where $k$ is possibly equal to zero (in which case $0\notin\partial \widetilde{H}_\alpha(u_\epsilon(t),x_\epsilon(t))$ for all $t \in (0,T)$). We have
\begin{subequations}
\begin{align}
    & \int_0^T \|\dot{x}_\epsilon(t)\|dt \\
    = & \sum_{i=0}^k \int_{t_i}^{t_{i+1}}\|\dot{x}_\epsilon(t)\|dt \label{eq:length_a}\\
    \leqslant & 2\left( \frac{c}{a} + \frac{\epsilon^2}{bc} \right) \sum_{i=0}^k \int_{t_i}^{t_{i+1}} \left| \frac{d}{dt}(\psi \circ \widetilde{H}_\alpha)(u_\epsilon(t),x_\epsilon(t)) \right|dt \label{eq:length_b}\\
     = & 2\left( \frac{c}{a} + \frac{\epsilon^2}{bc} \right) \sum_{i=0}^k \left| (\psi \circ \widetilde{H}_\alpha)(u_\epsilon(t_{i+1}),x_\epsilon(t_{i+1})) - (\psi \circ \widetilde{H}_\alpha)(u_\epsilon(t_i),x_\epsilon(t_i)) \right| \label{eq:length_c}\\
    \leqslant & 2\left( \frac{c}{a} + \frac{\epsilon^2}{bc} \right) \sum_{i=0}^k \psi \left( \left| \widetilde{H}_\alpha(u_\epsilon(t_{i+1}),x_\epsilon(t_{i+1})) - \widetilde{H}_\alpha(u_\epsilon(t_i),x_\epsilon(t_i)) \right|\right) \label{eq:length_d}\\
    = & 2\left( \frac{c}{a} + \frac{\epsilon^2}{bc} \right) \sum_{i=0}^k \psi \left( H_\alpha(u_\epsilon(t_i),x_\epsilon(t_i)) - H_\alpha(u_\epsilon(t_{i+1}),x_\epsilon(t_{i+1}))\right) \label{eq:length_e}\\
    \leqslant & 2\left( \frac{c}{a} + \frac{\epsilon^2}{bc} \right) (k+1) ~ \psi \left( \frac{1}{k+1} \sum_{i=0}^k H_\alpha(u_\epsilon(t_i),x_\epsilon(t_i)) - H_\alpha(u_\epsilon(t_{i+1}),x_\epsilon(t_{i+1}))\right) \label{eq:length_f}\\
     = & 2\left( \frac{c}{a} + \frac{\epsilon^2}{bc} \right) (k+1) ~ \psi \left(  \frac{H_\alpha(u_\epsilon(0),x_\epsilon(0)) - H_\alpha(u_\epsilon(T),x_\epsilon(T))}{k+1}\right) \label{eq:length_g}\\
    \leqslant & 4m \left( \frac{c}{a} + \frac{\epsilon^2}{bc} \right) ~ \psi \left(  \frac{H_\alpha(u_\epsilon(0),x_\epsilon(0)) - H_\alpha(u_\epsilon(T),x_\epsilon(T))}{2m}\right). \label{eq:length_h}
\end{align}
\end{subequations}
Indeed, \eqref{eq:length_c} is due to the fact that $\frac{d}{dt}(\psi \circ \widetilde{H}_\alpha)(u_\epsilon(t),x_\epsilon(t))>0$ for all $t \in (t_i,t_{i+1})$ or $\frac{d}{dt}(\psi \circ \widetilde{H}_\alpha)(u_\epsilon(t),x_\epsilon(t))<0$ for all $t \in (t_i,t_{i+1})$. \eqref{eq:length_d} holds because $\psi$ is concave and $\psi(0) = 0$. In particular, if $0\leqslant \lambda \leqslant \mu$, then $\psi(\mu) - \psi(\lambda) \leqslant \psi(\mu-\lambda)-\psi(0)$. \eqref{eq:length_e} holds because $0\notin\partial \widetilde{H}_\alpha(u_\epsilon(t),x_\epsilon(t))$ for all $t\in (t_i,t_{i+1})$. \eqref{eq:length_f} is due to Jensen's inequality. \eqref{eq:length_g} is the result of a telescoping sum. Finally, \eqref{eq:length_h} holds because $k \leqslant 2m-1$ and $\psi$ is concave and $\psi(0) = 0$. In particular, if $0 \leqslant \kappa$ and $0\leqslant \lambda \leqslant \mu$, then $\lambda \psi(\kappa/\lambda) \leqslant \mu \psi (\kappa/\mu)$. 

We next bound the argument of $\psi$ in \eqref{eq:length_h}. Let $\|\cdot\|_*$ denote the dual norm of $\|\cdot\|$. A Taylor bound yields
\begin{align*}
    & H_\alpha(u_\epsilon(0),x_\epsilon(0)) - H_\alpha(u_\epsilon(T),x_\epsilon(T)) \\[2mm]
    = & f(x_\epsilon(0)+\beta\dot{x}_\epsilon(0)) - f(x_\epsilon(T)+\beta\dot{x}_\epsilon(T)) + \alpha \beta^2 (\|\dot{x}_\epsilon(0)\|^2-\|\dot{x}_\epsilon(T)\|^2) \\
    \leqslant & f(x_\epsilon(0)) + \langle \nabla f(x_\epsilon(0)), \beta\dot{x}_\epsilon(0)\rangle + \frac{L}{2}\|\beta\dot{x}_\epsilon(0)\|^2 - f(x_\epsilon(T)) \\
    &- \langle \nabla f(x_\epsilon(T)), \beta\dot{x}_\epsilon(T)\rangle  + \frac{L}{2}\|\beta\dot{x}_\epsilon(T)\|^2 + \alpha \beta^2 \|\dot{x}_\epsilon(0)\|^2 \\ 
    \leqslant & f(x_\epsilon(0)) + \beta \| \nabla f(x_\epsilon(0)) \|_* \|\dot{x}_\epsilon(0)\| + \frac{L\beta^2}{2}\|\dot{x}_\epsilon(0)\|^2 - f(x_\epsilon(T)) \\
    & +\beta \|\nabla f(x_\epsilon(T))\|_*  \|\dot{x}_\epsilon(T)\| + \frac{L\beta^2}{2}\|\dot{x}_\epsilon(T)\|^2 + \alpha\beta^2 \|\dot{x}_\epsilon(0)\|^2 \\
    = & f(x_\epsilon(0)) - f(x_\epsilon(T)) + 2 \beta L r + L\beta^2 r^2 + \alpha \beta^2 r^2 \\[2mm]
    = & f(x_\epsilon(0)) - f(x_\epsilon(T)) + 2 \beta L r + L\beta^2 r^2 + (\beta\gamma+\epsilon)r^2/2 \\[2mm]
    \leqslant & f(x_\epsilon(0)) - f(x_\epsilon(T)) + 2 \epsilon r + \epsilon r^2 + (\epsilon+\epsilon) r^2/2\\[2mm]
    = & f(x_\epsilon(0)) - f(x_\epsilon(T)) + 2r(r+1)\epsilon
\end{align*}
by further imposing that $\beta \leqslant \epsilon/((1+\gamma)L)$. Indeed, $\beta L \leqslant \epsilon/(1+\gamma) \leqslant \epsilon$, $L\beta^2 \leqslant \epsilon \beta \leqslant \epsilon$ (since $\epsilon \leqslant \bar{\epsilon} \leqslant L$), and $\beta \gamma \leqslant \beta (1+\gamma) \leqslant \epsilon/L \leqslant \epsilon$ (since $1\leqslant L$). 

Finally, we seek to find a bound on the coefficient $c/a + \epsilon^2/(bc)$ in front of $\psi $ in \eqref{eq:length_h}. We would like it to be independent of $\epsilon$, otherwise the coefficient could blow up as $\epsilon$ approaches zero. By recalling that $\beta \leqslant \epsilon/((1+\gamma)L) \leqslant \epsilon/L$ and further imposing that $\beta \leqslant \sqrt{1+\gamma/L}-1$, we find that
\begin{align*}
    \frac{c}{a} &= \frac{L\beta+\gamma+\epsilon+\gamma\beta}{\gamma-L\beta(1+\beta/2)} \leqslant \frac{\gamma+2\epsilon+\gamma\epsilon/L}{\gamma-L\beta(1+\beta/2)} \leqslant \frac{\gamma+2\epsilon+\gamma\epsilon/L}{\gamma(1-1/2)} \\
    &= \frac{2\gamma L+ 4L\epsilon+2\gamma\epsilon}{\gamma L} \leqslant 4\left(1+\frac{L}{\gamma}\right). 
\end{align*}

In order to prevent the second term from blowing up, it will be convenient to choose $\beta$ as large as possible, namely by setting $\beta = \min \{ \epsilon/((1+\gamma)L) ,  \sqrt{1+\gamma/L}-1 \}$. This yields
\begin{subequations}
    \begin{align}
    \frac{\epsilon^2}{bc} & = \frac{\epsilon^2}{\beta(\epsilon-L\beta/2)(L\beta+\gamma+\epsilon+\gamma\beta)} \label{eq:second_a} \\
    & \leqslant \frac{2\epsilon}{\beta(L\beta+\gamma+\epsilon+\gamma\beta)} \label{eq:second_b}\\
    & \leqslant \max \left\{ \frac{2\epsilon}{\frac{\epsilon}{(1+\gamma)L}\left(L\frac{\epsilon}{(1+\gamma)L}+\gamma+\epsilon+\gamma\frac{\epsilon}{(1+\gamma)L}\right)} , \frac{2\epsilon}{\frac{\gamma}{3L}\left(L\frac{\gamma}{3L}+\gamma+\epsilon+\gamma\frac{\gamma}{3L}\right)} \right\} \label{eq:second_c} \\
    & \leqslant \max \left\{ \frac{2(1+\gamma)L}{\gamma} , \frac{6L}{\gamma} \right\} \label{eq:second_d} \\
    & \leqslant 2L+\frac{6L}{\gamma} \label{eq:second_e}
    \end{align}
\end{subequations}
Indeed, \eqref{eq:second_a} follows from the definitions of $a$ and $b$. \eqref{eq:second_b} is due to $\beta \leqslant \epsilon/L$. The first argument of the maximum in \eqref{eq:second_c} corresponds to the case where $\beta = \epsilon/((1+\gamma)L)$, while the second argument of the maximum corresponds to the case where $\beta = \sqrt{1+\gamma/L}-1$, in which case $\sqrt{1+\gamma/L}-1 \leqslant \epsilon/((1+\gamma)L) \leqslant \epsilon/L$. This implies that $1+\gamma/L \leqslant (1+\epsilon/L)^2$ and $\gamma/L \leqslant (\epsilon/L)^2 + 2\epsilon/L \leqslant 3$. Since $t \in \mathbb{R} \mapsto \sqrt{1+t}-1$ is concave, we find that $\beta \geqslant \gamma/(3L)$. In \eqref{eq:second_d} we discard all but one term in the sum in the denominators. \eqref{eq:second_e} provides a simpler bound devoid of a maximum. 

We conclude that \eqref{eq:length_formula} holds with
\begin{equation}
\label{eq:eta_varphi}
    \eta := 2r(r+1) ~~~\text{and}~~~ \varphi(t) := 8(2 + L + 5 L/\gamma)m ~ \psi\left(\frac{t}{2m}\right).
\end{equation}
\end{proof}

\section{Proof of Theorem \ref{thm:bounded}}
\label{sec:Proof of Theorem thm:bounded}

($\Longrightarrow$) Let $x_0 \in \mathbb{R}^n$ and $x$ be a global solution to \eqref{eq:P_0}. Let $T \geqslant 0$. By Lemma \ref{lemma:tracking}, there exists a sequence $\epsilon_k \searrow 0$ and global solutions $x_{\epsilon_k}$ to \eqref{eq:P_eps} initialized at $\{x_0\} \times \{0\}$ such that $\|x_{\epsilon_k}(t)-x(t)\| \leqslant 1/(k+1)$ for all $t\in[0,T]$. Since there exists a compact set $X \subset \mathbb{R}^n$ such that $x_{\epsilon_k}(t) \in X$ for all $t \geqslant 0$ and $k\in \mathbb{N}$, taking the limit yields that $x(t) \in X$ for all $t\in [0,T]$. As $T$ is arbitrary, we conclude that $x(t) \in X$ for all $t\geqslant 0$.

($\Longleftarrow$) We assume that for all $x_0 \in \mathbb{R}^n$, the global solution to \eqref{eq:P_0} is bounded. We seek to show that for all $(x_0,\dot{x}_0) \in \mathbb{R}^n \times \mathbb{R}^n$, the global solution to \eqref{eq:P_eps} is uniformly bounded for all sufficiently small $\epsilon>0$. We will actually show something slightly stronger. Following the recent work \cite{josz2023global} on gradient dynamics, we will show by induction that the length is uniformly bounded. Let $X_0$ be a nonempty compact subset of $\mathbb{R}^n$ and let $r_0 \geqslant 0$. We will prove that there exists $\overline{\epsilon}>0$ such that $\sigma(X_0,r_0,\bar{\epsilon})<\infty$ where
\begin{subequations}
\label{eq:sup_continuous_X_0}
\begin{align}
    \sigma(X_0,r_0,\bar{\epsilon}) := & \sup\limits_{\begin{array}{c}\text{\footnotesize{$x_\epsilon\in C^2(\mathbb{R}_+,\mathbb{R}^n)$}} \\ \text{\footnotesize{$\epsilon \in (0,\overline{\epsilon}]$}} \end{array}} ~~ \int_0^\infty \|\dot{x}_\epsilon(t)\|dt \\
   & ~~ \mathrm{s.t.} ~~
\left\{ 
\begin{array}{l}
\epsilon\ddot{x}_\epsilon(t) + \gamma\dot{x}_\epsilon(t) + \nabla f(x_\epsilon(t)) = 0, ~ \forall t\geqslant 0,\\[2mm] x_\epsilon(0) \in X_0,\quad \dot{x}_\epsilon(0) \in B(0,r_0).
\end{array}
\right.
\end{align}
\end{subequations}

Let $\Phi:\mathbb{R}_+\times\mathbb{R}^{n}\rightarrow\mathbb{R}^n$ be the gradient flow of $f$ defined for all $(t,x_0)\in \mathbb{R}_+\times\mathbb{R}^{n}$ by $\Phi(t,x_0)=x(t)$ where $x$ is the solution to \eqref{eq:P_0}. Let $\Phi_0:=\Phi(\mathbb{R}_+,X_0)$ and let $C$ be the set of critical points of $f$ in $\overline{\Phi}_0$. Note that $C$ is compact by \cite[Lemma 1]{josz2023global} and \cite[2.1.5 Proposition p. 29]{clarke1990}. Thus there exists $\xi>0$ such that either $X_0\subset C$ or $X_0\setminus\mathring{B}(C,\xi/4)\neq\emptyset$ where $\mathring{B}(C,\xi/4):=C+\mathring{B}(0,\xi/4)$. 

By Lemma \ref{lemma:speed}, there exists $r_1>0$ such that for all $\epsilon,T>0$, if $x_\epsilon:[0,T]\rightarrow B(\overline{\Phi}_0,\xi)$ is a solution to \eqref{eq:P_eps} such that $\|x_\epsilon(0)\| \leqslant r_0$, then $\|\dot{x}_\epsilon(t)\| \leqslant r_1$ for all $t\in [0,T]$. By Lemma \ref{lemma:length}, there exist $\eta>0$ and a diffeomorphism $\varphi:\mathbb{R}_+\rightarrow\mathbb{R}_+$ such that for all $\epsilon\in(0,1]$   and all $T\geqslant 0$, if $x_\epsilon:[0,T]\rightarrow B(\overline{\Phi}_0,\xi)$ is a solution to \eqref{eq:P_eps} such that $\|x_\epsilon(t)\| \leqslant r_1$ for all $t\in [0,T]$, then
\begin{equation}\label{eq:HBFfinitelength-basic}
    \int_0^T\|\dot{x}_\epsilon(t)\|dt \leqslant \varphi\left(f(x_\epsilon(0))-f(x_\epsilon(T))+\eta\epsilon\right). 
\end{equation}
Since $f$ is continuous, there exists $\delta\in(0,\xi/2)$ such that 
\begin{equation}\label{eq:HBFcltocritval}
    f(x) - \max_C f \leqslant \frac{1}{4}\varphi^{-1}\left(\frac{\xi}{2}\right), \quad \forall x\in B(C,\delta).
\end{equation}

We next show that there exists $\epsilon_0>0$ such that for all $\epsilon\in(0,\epsilon_0]$, there exists $t^*\geqslant 0$ such that $x_\epsilon(t^*)\in B(C,\delta)$. If $X_0\subset C$, then this is guaranteed by taking $\epsilon_0 = 1$ and $t^*=0$. If $X_0\setminus\mathring{B}(C,\xi/4)\neq\emptyset$, then  $\overline{\Phi}_0\setminus\mathring{B}(C,\delta/2)$ is nonempty since it contains $X_0\setminus\mathring{B}(C,\xi/4)$. Hence $\|\nabla f\|/\gamma$ attains its infimum $\nu$ on the compact set $\overline{\Phi}_0\setminus\mathring{B}(C,\delta/2)$. It must be that $\nu>0$ because $\overline{\Phi}_0\setminus\mathring{B}(C,\delta/2)$ is devoid of critical points of $f$. It thus makes sense to define $T:=2\sigma(X_0)/\nu$ where 
\begin{subequations}
\label{eq:gradient_flow}
\begin{align}
    \sigma(X_0) := & \sup\limits_{\text{\footnotesize $x \in C^1(\mathbb{R}_+,\mathbb{R}^n)$}} ~~ \int_0^\infty \|\dot{x}(t)\|dt \\
   & ~~ \mathrm{s.t.} ~~
\left\{ 
\begin{array}{l}
\gamma\dot{x}(t) + \nabla f(x(t)) = 0, ~ \forall t\geqslant 0,\\[2mm] x(0) \in X_0,
\end{array}
\right.
\end{align}
\end{subequations}
is finite by \cite[Lemma 1]{josz2023global}. Since $X_0\not\subset C$, it holds that $\sigma(X_0)>0$ and $T>0$. By Lemma \ref{lemma:tracking}, there exists $\epsilon_0\in (0,\min\{1,\varphi^{-1}(\xi/2)/(4\eta)\}]$ such that for all $\epsilon\in (0,\epsilon_0]$ and for any feasible point $(x_\epsilon,\epsilon)$ of \eqref{eq:sup_continuous_X_0}, there exists a feasible point $x$ of \eqref{eq:gradient_flow} such that $\|x_\epsilon(t) - x(t)\| \leqslant \delta/2$ for all $t\in[0,T]$. For any such $x$ there exists $t^*\in(0,T)$ such that $\|\dot{x}(t^*)\|<\nu$, otherwise $\sigma(X_0) < T(2\sigma(X_0))/T = T\nu\leqslant \int_0^{T}\|\dot{x}(t)\|dt \leqslant \sigma(X_0)$. Thus $\|\nabla f(x(t^*))\|/\gamma=\|\dot{x}(t^*)\|<\nu$. Since $x(t^*)\in\overline{\Phi}_0$, by definition of $\nu$, it follows that $x(t^*)\in \mathring{B}(C,\delta/2)$. Hence there exists $x^*\in C$ such that $\|x(t^*)-x^*\|\leqslant\delta/2$ and $\|x_\epsilon(t^*)-x^*\|\leqslant \|x_\epsilon(t^*)-x(t^*)\|+\|x(t^*)-x^*\| \leqslant \delta/2+\delta/2=\delta$. In other words, $x_\epsilon(t^*) \in B(C,\delta)$.

Fix $\epsilon \in (0,\epsilon_0]$ and let $(x_\epsilon,\epsilon)$ be a feasible point of \eqref{eq:sup_continuous_X_0}. By the previous paragraph, there exists $t^* \geqslant 0$ such that $x_\epsilon(t^*) \in B(C,\delta)$. Let $T^*=\inf\{t\geqslant t^*:x_\epsilon(t)\not\in\mathring{B}(C,\xi)\}$. If $T^* = \infty$, then $x_\epsilon(t) \in \mathring{B}(C,\xi)$ for all $t \geqslant t^*$. Since $x_\epsilon(t) \in B(\Phi_0,\delta/2)$ for all $t\in [0,t^*]$ and $\delta < \xi/2$, it follows that $x_\epsilon(t)\in B(\overline{\Phi}_0,\xi)$ for all $t\geqslant 0$.
The length formula \eqref{eq:HBFfinitelength-basic} then yields
 \begin{equation*}
    \int_{0}^\infty \|\dot{x}_\epsilon(t)\|dt \leqslant \varphi\left(\sup_{X_0} f -\min_{B(\overline{\Phi}_0,\xi)} f+\eta\epsilon\right). 
\end{equation*}
If $T^* < \infty$, then let $x^*\in C$ be such that $x_\epsilon(t^*) \in B(x^*,\delta)$ and observe that
\begin{equation*}
    \int_{t^*}^{T^*} \|x'_\epsilon(t)\|dt \geqslant \|x_\epsilon(T^*) - x_\epsilon(t^*)\| \geqslant \|x_\epsilon(T^*)-x^*\|-\|x_\epsilon(t^*)-x^*\| \geqslant \xi -\delta \geqslant \frac{\xi}{2}.
\end{equation*}
By the length formula \eqref{eq:HBFfinitelength-basic}, we have
\begin{equation*}
    \frac{\xi}{2} \leqslant \int_{t^*}^{T^*} \|x'_\epsilon(t)\|dt \leqslant \varphi\left(f(x_\epsilon(t^*)) - f(x_\epsilon(T^*)) +\eta\epsilon\right)
\end{equation*}
Composing by $\varphi^{-1}$, we find that $\varphi^{-1}(\xi/2) \leqslant f(x_\epsilon(t^*)) - f(x_\epsilon(T^*)) +\eta\epsilon$. It follows that 
\begin{align*}
    f(x_\epsilon(T^*)) & \leqslant f(x_\epsilon(t^*)) - \varphi^{-1}(\xi/2) +\eta\epsilon \\[2mm]
    &\leqslant \max_C f + \varphi^{-1}(\xi/2)/4 - \varphi^{-1}(\xi/2) + \varphi^{-1}(\xi/2)/4 \\
    & = \max_C f - \varphi^{-1}(\xi/2)/2
\end{align*}
where we use the bound in \eqref{eq:HBFcltocritval} and the fact that $\epsilon \leqslant \epsilon_0 \leqslant \varphi^{-1}(\xi/2)/(4\eta)$. In other words, $x_\epsilon(T^*)$ belongs to the set 
\begin{equation*}
    X_1 := \left\{x\in B(C,\xi):f(x)\leqslant \max_{C} f-\frac{1}{2}\varphi^{-1}\left(\frac{\xi}{2}\right)\right\}.
\end{equation*}
Since $(x_\epsilon(t),\dot{x}_\epsilon(t))\in B(\overline{\Phi}_0,\xi)\times B(0,r_1)$ for all $t\in[0,T^*]$, by definition of $\sigma$ in \eqref{eq:sup_continuous_X_0} we have
\begin{align*}
    \int_0^\infty\|\dot{x}_\epsilon(t)\|dt &= \int_0^{T^*}\|\dot{x}_\epsilon(t)\|dt + \int_{T^*}^\infty\|\dot{x}_\epsilon(t)\|dt \\
    &\leqslant \varphi\left(f(x_\epsilon(0))-f(x_\epsilon(T^*))+\eta\epsilon\right) + \sigma(X_1,r_1,\epsilon)
\end{align*}
Combining the cases when $T^*<\infty$ and $T^*=\infty$, one further concludes that 
\begin{equation*}
    \sigma(X_0,r_0,\epsilon) \leqslant \varphi\left(\sup\limits_{X_0} f-\inf\limits_{B(\overline{\Phi}_0,\xi)}f+\eta\epsilon\right)+\max\{\sigma(X_1,r_1,\epsilon),0\}
\end{equation*}
for all $\epsilon\in(0,\epsilon_0]$.

It now suffices to treat $(X_1,r_1)$ as the new initial conditions and reason by induction. For notational convenience, let $\varphi_0 := \varphi$, $\xi_0 := \xi$, and $\eta_0 := \eta$. Suppose that at iteration $k\in\mathbb{N}$ we obtain 
\begin{equation}\label{eq:HBFrecursivek}
    \sigma(X_{k},r_{k},\epsilon) \leqslant \varphi_{k}\left(\sup\limits_{X_{k}} f-\inf\limits_{B(\overline{\Phi}_{k},\xi_{k})}f+\eta_{k}\epsilon\right)+\max\{\sigma(X_{k+1},r_{k+1},\epsilon),0\}
\end{equation}
for all $\epsilon\in(0,\epsilon_{k}]$. Since 
\begin{equation*}
    f(\Phi(t,x_{k+1})) \leqslant f(\Phi(0,x_{k})) \leqslant \max_{C_{k}} f-\frac{1}{2}\varphi_{k}^{-1}\left(\frac{\xi_{k}}{2}\right) < \max_{C_{k}} f
\end{equation*}
for all $x_{k+1}\in X_{k+1}$ and $t\geqslant 0$, the maximal critical value of $f$ in $\overline{\Phi}_{k+1}$ is less than the maximal critical value of $f$ in $\overline{\Phi}_{k}$. By the definable Morse-Sard theorem \cite[Corollary 9]{bolte2007clarke}, $f$ has finitely many critical values. Hence there exists $K\geqslant 1$ such that $X_K=\emptyset$ and $\sigma(X_K,r_K,\epsilon_K)=-\infty$ by convention. Let $\bar{\epsilon}:=\min\{\epsilon_0,\ldots,\epsilon_{K-1}\}>0$. Then \eqref{eq:HBFrecursivek} holds for $k=1,\ldots,K$ where $\epsilon:=\bar{\epsilon}$. We conclude that 
\begin{equation*}
    \sigma(X_0,r_0,\bar{\epsilon}) \leqslant \sum_{k=0}^{K-1}\varphi_{k}\left(\sup\limits_{X_k} f-\inf\limits_{B(\overline{\Phi}_{k},\xi_{k})}f+\eta_{k}\bar{\epsilon}\right) < \infty.
\end{equation*}

\section{An application of Theorem \ref{thm:bounded}}
\label{sec:application}

From the example in Section \ref{sec:Example}, we already know that the global solution of \eqref{eq:P_eps} may not converge uniformly over $\mathbb{R}_+$ to the global solution of \eqref{eq:P_0} with the same initial point. However, with the slightly stronger version of Theorem \ref{thm:bounded} (see the comments right above \eqref{eq:sup_continuous_X_0}, which allows the boundedness of solution to be uniform over any compact set of initial points), we are able to deduce that the global solution of \eqref{eq:P_eps} converges uniformly over $[t_0,\infty)$ to some global solution of \eqref{eq:P_0} with  a possibly different initial point, by evoking the geometric singular perturbation theory (GSP) \cite{fenichel1979geometric,arnold1995geometric,kuehn2015multiple}. 

For all $\epsilon>0$, let $\Phi^\epsilon:\mathbb{R}_+\times\mathbb{R}^{n}\times\mathbb{R}^n\rightarrow\mathbb{R}^n$ be defined for all $(t,x_0,\dot{x}_0)\in \mathbb{R}_+\times\mathbb{R}^{n}\times\mathbb{R}^n$ by $\Phi^\epsilon(t,x_0,\dot{x}_0):=x_\epsilon(t)$ where $x_\epsilon$ is the global solution to \eqref{eq:P_eps} with initial point $(x_0,\dot{x}_0)$. Similarly, let $\Phi:\mathbb{R}_+\times\mathbb{R}^{n}\rightarrow\mathbb{R}^n$ be defined for all $(t,x_0)\in \mathbb{R}_+\times\mathbb{R}^{n}$ by $\Phi(t,x_0):=x(t)$ where $x$ is the global solution to \eqref{eq:P_0} with initial point $x_0$. We also denote $\dot{\Phi}^\epsilon(t,x_0,\dot{x}_0)=\dot{x}_\epsilon(t)$ and $\dot{\Phi}(t,x_0)=\dot{x}(t)$ accordingly. The uniform convergence result described in the above paragraph is given as follows. 
\begin{corollary}
    \label{cor:uniconv}
    Let $\gamma>0$ and $f:\mathbb{R}^n\rightarrow\mathbb{R}$ be a $C^{1,1}_\mathrm{loc}$ lower bounded function definable in an o-minimal structure on the real field. If for all $x_0 \in \mathbb{R}^n$, $\Phi(\cdot,x_0)$ is bounded, then for all $(x_0,\dot{x}_0) \in \mathbb{R}^n \times \mathbb{R}^n$ and all $t_0>0$, there exists $x_0'\in\mathbb{R}^n$ such that $\Phi^\epsilon(t,x_0,\dot{x}_0)\to \Phi(t,x_0')$ uniformly over $t\in[t_0,\infty)$ as $\epsilon\searrow 0$. 
\end{corollary}
\begin{proof}
    Fix any $(x_0,\dot{x}_0)\in \mathbb{R}^n \times \mathbb{R}^n$. By Theorem \ref{thm:bounded}, there exists $\epsilon_0>0$ and $c_1>0$ such that $\|\Phi^\epsilon(t,x_0,\dot{x}_0)\| \leqslant c_1$ and $\|\dot{\Phi}^\epsilon(t,x_0,\dot{x}_0)\| \leqslant c_1$ for all $t\geqslant 0$ and all $\epsilon\in(0,\epsilon_0]$. By the slightly stronger version of Theorem \ref{thm:bounded}, there exists $\epsilon_1\in(0,\epsilon_0]$ and $c_2>0$ such that $\|\Phi^\epsilon(t,x_0',\dot{x}_0')\|\leqslant c_2$ and $\|\dot{\Phi}^\epsilon(t,x_0',\dot{x}_0')\|\leqslant c_2$ for all $t\geqslant 0$, $x_0',\dot{x}_0'\in B(0,c_1+1)$ and $\epsilon\in(0,\epsilon_1]$. Thus, we can choose a closed ball $K$ such that its interior $\text{int}\,K$ satisfies 
    \begin{equation*}
        \text{int}\,K \supseteq \{\gamma\Phi^\epsilon(t,x_0',\dot{x}_0')+\epsilon\dot{\Phi}^\epsilon(t,x_0',\dot{x}_0'):t\geqslant 0, x_0',\dot{x}_0'\in B(0,c_1+1),\epsilon\in(0,\epsilon_1]\}. 
    \end{equation*}
    Consider the system 
    \begin{equation}
    \label{eq:S_eps} \tag{$S_\epsilon$}
        \begin{cases}
            \epsilon\dot{x}_\epsilon(t) = -\gamma x_\epsilon(t) + y_\epsilon(t), \\
            \dot{y}_\epsilon(t) = -\nabla f(x_\epsilon(t)), 
        \end{cases}~~~ \forall t \geqslant 0. 
    \end{equation}
    Similar to $\Phi^\epsilon$, define $\Psi_x^\epsilon,\Psi_y^\epsilon:\mathbb{R}_+\times\mathbb{R}^n\times\mathbb{R}^n\rightarrow \mathbb{R}^n$ by $(\Psi_x^\epsilon(t,x_0',y_0'),\Psi_x^\epsilon(t,x_0',y_0')):=(x_\epsilon(t),y_\epsilon(t))$ as the solution to \eqref{eq:S_eps} with initial point $(x_0',y_0')$. Then 
    \begin{align*}
        \Psi_x^\epsilon(t,x_0,\gamma x_0+\epsilon\dot{x}_0) &= \Phi^\epsilon(t,x_0,\dot{x}_0), \\
        \Psi_y^\epsilon(t,x_0,\gamma x_0+\epsilon\dot{x}_0) &= \gamma\Phi^\epsilon(t,x_0,\dot{x}_0)+\epsilon\dot{\Phi}^\epsilon(t,x_0,\dot{x}_0). 
    \end{align*}
    Let $\widetilde{K}:=K+B(0,1)$, and apply GSP theory \cite[Theorem 2]{arnold1995geometric} to the compact manifold with boundary
    \begin{equation*}
        Z_0 = \{(x,y)\in\mathbb{R}^n\times \widetilde{K}:x=x^*(y)= y/\gamma\},
    \end{equation*}
    there exists $\epsilon_2\in(0,\epsilon_1]$ such that for $\epsilon\in(0,\epsilon_2]$, there exists a locally invariant manifold
    \begin{equation*}
        Z_\epsilon = \{(x,y)\in\mathbb{R}^n\times \widetilde{K}: x = \bar{x}(y,\epsilon) = x^*(y) + O(\epsilon)\}. 
    \end{equation*}
    Restricted to $Z_\epsilon$, \eqref{eq:S_eps} reduces to  
    \begin{equation}
    \label{eq:S_eps^0} \tag{$S_\epsilon^0$}
        \dot{y}_\epsilon^0 = -\nabla f(\bar{x}(y_\epsilon^0,\epsilon)) = -\nabla f(x^*(y_\epsilon^0)) + O(\epsilon). 
    \end{equation}
    
    Given that $Z_0$ is uniformly asymptotically stable, we know that $Z_\epsilon$ is locally asymptotically stable and hence any solution to \eqref{eq:S_eps} with initial condition close to $Z_0$ will converge to a solution of \eqref{eq:S_eps^0}. More precisely, by \cite[Corollary 1]{arnold1995geometric}, there exists $v\in \widetilde{K}$, $C,\alpha>0$ and $\epsilon_3\in(0,\epsilon_2]$ such that for all $\epsilon\in(0,\epsilon_3]$, 
    \begin{subequations}\label{eq:exp_decay_esti}
        \begin{align}
            \|\Psi^\epsilon_x(t,x_0,\gamma x_0+\epsilon\dot{x}_0)-\Psi^\epsilon_x(t,\bar{x}(v,\epsilon),v)\| &\leqslant Ce^{-\alpha t/\epsilon}, \\
            \|\Psi^\epsilon_y(t,x_0,\gamma x_0+\epsilon\dot{x}_0)-\Psi^\epsilon_y(t,\bar{x}(v,\epsilon),v)\| &\leqslant Ce^{-\alpha t/\epsilon}
        \end{align}
    \end{subequations}
    for all $t$ such that $\Psi^\epsilon_y(t,\bar{x}(v,\epsilon),v)\in \widetilde{K}$. Let $Z_\epsilon^K:=Z_\epsilon\cap (\mathbb{R}^n\times K)$.  Since $\gamma x_0\in\text{int}\,K$, there exists $\epsilon_4\in (0,\epsilon_3]$ such that $(x_0,\gamma x_0+\epsilon\dot{x}_0)\in W^s(Z_\epsilon^K)=\cup_{v\in K}W^s((\bar{x}(v,\epsilon),v))$ for all $\epsilon\in(0,\epsilon_4]$, where $W^s$ denotes the stable manifold defined in the GSP theory. This means \eqref{eq:exp_decay_esti} actually holds for some $v\in K$. 
    Thus, by continuity, there exists $t_1>0$ such that $\Psi^\epsilon_y(t,\bar{x}(v,\epsilon),v)\in \widetilde{K}$ for $t\in[0,t_1]$. 
    
    Next we show that actually $\Psi^\epsilon_y(t,\bar{x}(v,\epsilon),v)\in \widetilde{K}$ for all $t\geq 0$. We claim that $\|\underbrace{\Psi^\epsilon_x(t_1,\bar{x}(v,\epsilon),v)}_{=:x_0'}\|\leqslant c_1+1$ and $\|\underbrace{(\Psi^\epsilon_y(t_1,\bar{x}(v,\epsilon),v)-\gamma x_0')/\epsilon}_{=:\dot{x}_0'} \| \leqslant c_1+1$.  
    If the claim is true, then by definition of $K$, 
    \begin{align*}
        \Psi^\epsilon_y(t,x_0',\gamma x_0'+\epsilon\dot{x}_0')=\gamma\Phi^\epsilon(t,x_0',\dot{x}_0')+\epsilon\dot{\Phi}^\epsilon(t,x_0',\dot{x}_0') \in K\subseteq \widetilde{K},\quad \forall t\geqslant 0. 
    \end{align*}
    This further implies  
    \begin{equation*}
        \Psi^\epsilon_y(t+t_1,\bar{x}(v,\epsilon),v) = \Psi^\epsilon_y(t,\Psi^\epsilon_x(t_1,\bar{x}(v,\epsilon),v),\Psi^\epsilon_y(t_1,\bar{x}(v,\epsilon),v)) \in \widetilde{K},\quad \forall t\geqslant 0.  
    \end{equation*}
    Thus, we prove that $\Psi^\epsilon_y(t,\bar{x}(v,\epsilon),v)\in \widetilde{K}$ for all $t\geq 0$. To verify our claim, by \eqref{eq:exp_decay_esti}, we have
    \begin{equation*}
        \|x_0'\|\leqslant \|\Phi^\epsilon(t_1,x_0,\dot{x}_0)\| + \|\Phi^\epsilon(t_1,x_0,\dot{x}_0)-x_0'\| \leqslant c_1+Ce^{-\alpha t_1/\epsilon}. 
    \end{equation*}
    Then it is clear that there exists $\epsilon_5\in(0,\epsilon_4]$ such that $\|x_0'\|\leqslant c_1+1$ for all $\epsilon\in(0,\epsilon_5]$. In addition, consider
    \begin{align*}
        \|\dot{\Phi}^\epsilon(t_1,x_0,\dot{x}_0)-\dot{x}_0'\| &\leqslant \|\dot{\Phi}^\epsilon(t_1,x_0,\dot{x}_0)-(\Psi^\epsilon_y(t,x_0,\gamma x_0+\epsilon\dot{x}_0)-\gamma x_0')/\epsilon\| + C\epsilon^{-1}e^{-\alpha t_1/\epsilon} \\
        &\leqslant \|\dot{\Phi}^\epsilon(t_1,x_0,\dot{x}_0)-(\gamma\Phi^\epsilon(t_1,x_0,\dot{x}_0)+\epsilon\dot{\Phi}^\epsilon(t_1,x_0,\dot{x}_0)-\gamma x_0')/\epsilon\| + C\epsilon^{-1}e^{-\alpha t_1/\epsilon} \\
        &= \gamma\epsilon^{-1}\|\Phi^\epsilon(t_1,x_0,\dot{x}_0)-x_0'\|+ C\epsilon^{-1}e^{-\alpha t_1/\epsilon} \\
        &\leqslant (1+\gamma)C\epsilon^{-1}e^{-\alpha t_1/\epsilon}. 
    \end{align*}
    Note that $C\epsilon^{-1}e^{-\alpha t_1/\epsilon}\to 0$ as $\epsilon\to 0$. Thus, we can find $\epsilon_6\in(0,\epsilon_5]$ so that $\|\dot{\Phi}^\epsilon(t_1,x_0,\dot{x}_0)-\dot{x}_0'\|\leqslant 1$ for all $\epsilon\in(0,\epsilon_6]$. This proves $\|\dot{x}_0'\|\leqslant c_1+1$ because we know that $\|\dot{\Phi}^\epsilon(t_1,x_0,\dot{x}_0)\|\leqslant c_1$. 
    
    From the previous result, for any $\delta>0$, there exists $\epsilon_7\in(0,\epsilon_6]$ such that for all $\epsilon\in(0,\epsilon_7]$ and all $t\geqslant t_0$, 
    \begin{equation}\label{eq:S_eps_approx_S_eps^0}
        \|(\Phi^\epsilon(t,x_0,\dot{x}_0),\gamma\Phi^\epsilon (t,x_0,{\dot{x}}_0)+\epsilon\dot{\Phi}^\epsilon(t,x_0,\dot{x}_0))-(\Psi^\epsilon_x(t,\bar{x}(v,\epsilon),v),\Psi^\epsilon_y(t,\bar{x}(v,\epsilon),v))\|\leqslant \frac{\delta}{2}. 
    \end{equation}
    Fenichel's GSP theory tells us that on manifold $Z_\epsilon$, \eqref{eq:S_eps^0} is a regular perturbation of the degenerate system 
    \begin{equation}
    \label{eq:S_0^0} \tag{$S_0^0$}
        \dot{y} = -\nabla f(x^*(y)) = -\nabla f(y/\gamma). 
    \end{equation}
    By using a simple change of variable $y\leftarrow y/\gamma$, it is easy to see \eqref{eq:S_0^0} is equivalent to the gradient system in $(P_0)$. This means any solution to \eqref{eq:S_eps^0} converges uniformly to a solution to $(P_0)$. Thus, for any $\delta>0$, there exists $\epsilon_8\in(0,\epsilon_7]$ such that for all $\epsilon\in (0,\epsilon_8]$,   
    \begin{equation}\label{eq:S_eps^0_approx_S_0^0}
        \|(\Psi^\epsilon_x(t,\bar{x}(v,\epsilon),v),\Psi^\epsilon_y(t,\bar{x}(v,\epsilon),v))-(\Phi(t,v),\gamma \Phi(t,v))\|\leqslant \frac{\delta}{2}, \quad \forall t\geqslant 0. 
    \end{equation}
    Combining \eqref{eq:S_eps_approx_S_eps^0} and \eqref{eq:S_eps^0_approx_S_0^0} would yield the desired results. Obviously, the initial point $v$ of the limiting solution is likely different from the initially chosen $x_0$. 
\end{proof}

\section*{Acknowledgments}
We thank the reviewers and the associate editor for their valuable feedback.

\bibliographystyle{abbrv}    
\bibliography{JDE/JDDE/sn-bibliography}

\begin{thebibliography}{10}

\bibitem{alvarez2000minimizing}
F.~Alvarez.
\newblock On the minimizing property of a second order dissipative system in
  hilbert spaces.
\newblock {\em SIAM Journal on Control and Optimization}, 38(4):1102--1119,
  2000.

\bibitem{arnold1995geometric}
L.~Arnold, C.~K. Jones, K.~Mischaikow, G.~Raugel, and C.~K. Jones.
\newblock Geometric singular perturbation theory.
\newblock {\em Dynamical Systems: Lectures Given at the 2nd Session of the
  Centro Internazionale Matematico Estivo (CIME) held in Montecatini Terme,
  Italy, June 13--22, 1994}, pages 44--118, 1995.

\bibitem{attouch2000heavy}
H.~Attouch, X.~Goudou, and P.~Redont.
\newblock The heavy ball with friction method, i. the continuous dynamical
  system: global exploration of the local minima of a real-valued function by
  asymptotic analysis of a dissipative dynamical system.
\newblock {\em Communications in Contemporary Mathematics}, 2(01):1--34, 2000.

\bibitem{aubin1984differential}
J.-P. Aubin and A.~Cellina.
\newblock {\em Differential inclusions: set-valued maps and viability theory},
  volume 264.
\newblock Springer-Verlag, Berlin, 1984.

\bibitem{begout2015damped}
P.~B{\'e}gout, J.~Bolte, and M.~A. Jendoubi.
\newblock On damped second-order gradient systems.
\newblock {\em Journal of Differential Equations}, 259(7):3115--3143, 2015.

\bibitem{bolte2007clarke}
J.~Bolte, A.~Daniilidis, A.~Lewis, and M.~Shiota.
\newblock Clarke subgradients of stratifiable functions.
\newblock {\em SIAM Journal on Optimization}, 18(2):556--572, 2007.

\bibitem{bolte2010characterizations}
J.~Bolte, A.~Daniilidis, O.~Ley, and L.~Mazet.
\newblock Characterizations of {\l}ojasiewicz inequalities: subgradient flows,
  talweg, convexity.
\newblock {\em Transactions of the American Mathematical Society},
  362(6):3319--3363, 2010.

\bibitem{boct2020second}
R.~I. Bo{\c{t}}, E.~R. Csetnek, and S.~C. L{\'a}szl{\'o}.
\newblock A second-order dynamical approach with variable damping to nonconvex
  smooth minimization.
\newblock {\em Applicable Analysis}, 99(3):361--378, 2020.

\bibitem{clarke1990}
F.~H. Clarke.
\newblock {\em Optimization and Nonsmooth Analysis}.
\newblock SIAM Classics in Applied Mathematics, 1990.

\bibitem{coddington1955theory}
E.~A. Coddington and N.~Levinson.
\newblock {\em Theory of ordinary differential equations}.
\newblock Tata McGraw-Hill Education, 1955.

\bibitem{fenichel1979geometric}
N.~Fenichel.
\newblock Geometric singular perturbation theory for ordinary differential
  equations.
\newblock {\em Journal of differential equations}, 31(1):53--98, 1979.

\bibitem{Fitzpatrick2010}
P.~M. Fitzpatrick and H.~L. Royden.
\newblock {\em Real Analysis}.
\newblock Pearson, Upper Saddle River, NJ, 4 edition, Jan. 2010.

\bibitem{gabrielov1996complements}
A.~Gabrielov.
\newblock Complements of subanalytic sets and existential formulas for analytic
  functions.
\newblock {\em Inventiones mathematicae}, 125(1):1--12, 1996.

\bibitem{haraux1986asymptotics}
A.~Haraux.
\newblock Asymptotics for some nonlinear ode of the second order.
\newblock {\em Nonlinear Analysis: Theory, Methods \& Applications},
  10(12):1347--1355, 1986.

\bibitem{hoppensteadt1966singular}
F.~C. Hoppensteadt.
\newblock Singular perturbations on the infinite interval.
\newblock {\em Transactions of the American Mathematical Society},
  123(2):521--535, 1966.

\bibitem{josz2023global}
C.~Josz.
\newblock Global convergence of the gradient method for functions definable in
  o-minimal structures.
\newblock {\em Mathematical Programming}, pages 1--29, 2023.

\bibitem{josz2023convergence}
C.~Josz, L.~Lai, and X.~Li.
\newblock Convergence of the momentum method for semi-algebraic functions with
  locally lipschitz gradients.
\newblock {\em arXiv preprint arXiv:2307.03331}, 2023.

\bibitem{kokotovic1999singular}
P.~Kokotovi{\'c}, H.~K. Khalil, and J.~O'reilly.
\newblock {\em Singular perturbation methods in control: analysis and design}.
\newblock SIAM, 1999.

\bibitem{kuehn2015multiple}
C.~Kuehn et~al.
\newblock {\em Multiple time scale dynamics}, volume 191.
\newblock Springer, 2015.

\bibitem{kurdyka1998gradients}
K.~Kurdyka.
\newblock On gradients of functions definable in o-minimal structures.
\newblock In {\em Annales de l'institut Fourier}, volume~48, pages 769--783,
  1998.

\bibitem{kurdyka43quasi}
K.~Kurdyka and A.~Parusiski.
\newblock {Quasi-convex decomposition in o-minimal structures. Application to
  the gradient conjecture. Singularity theory and its applications, 137177}.
\newblock {\em Adv. Stud. Pure Math}, 43, 2006.

\bibitem{law1965ensembles}
S.~\L{}ojasiewicz.
\newblock Ensembles semi-analytiques.
\newblock {\em IHES notes}, 1965.

\bibitem{popa2002differential}
D.~Popa and N.~Lungu.
\newblock On some differential inequalities.
\newblock In {\em Seminar on Fixed Point Theory, Cluj-Napoca}, volume~3, pages
  323--326, 2002.
\newblock
  \url{http://www.math.ubbcluj.ro/~nodeacj/download.php?f=020POPA2.pdf}.

\bibitem{rudin1964principles}
W.~Rudin et~al.
\newblock {\em Principles of mathematical analysis}, volume~3.
\newblock McGraw-hill New York, 1964.

\bibitem{santambrogio2017euclidean}
F.~Santambrogio.
\newblock $\{$Euclidean, metric, and Wasserstein$\}$ gradient flows: an
  overview.
\newblock {\em Bulletin of Mathematical Sciences}, 7:87--154, 2017.

\bibitem{seidenberg1954new}
A.~Seidenberg.
\newblock A new decision method for elementary algebra.
\newblock {\em Annals of Mathematics}, pages 365--374, 1954.

\bibitem{tarski1951decision}
A.~Tarski.
\newblock {A decision method for elementary algebra and geometry: Prepared for
  publication with the assistance of JCC McKinsey}.
\newblock 1951.

\bibitem{vasil1995boundary}
A.~B. Vasil'Eva, V.~F. Butuzov, and L.~V. Kalachev.
\newblock {\em The boundary function method for singular perturbation
  problems}.
\newblock SIAM, 1995.

\bibitem{wilkie1996model}
A.~J. Wilkie.
\newblock Model completeness results for expansions of the ordered field of
  real numbers by restricted pfaffian functions and the exponential function.
\newblock {\em Journal of the American Mathematical Society}, 9(4):1051--1094,
  1996.

\bibitem{zavriev1993heavy}
S.~Zavriev and F.~Kostyuk.
\newblock Heavy-ball method in nonconvex optimization problems.
\newblock {\em Computational Mathematics and Modeling}, 4(4):336--341, 1993.

\end{thebibliography}

\end{document}